\documentclass[letterpaper]{amsart}

\usepackage{latexsym,amssymb}
\usepackage{amsmath}
\usepackage{amsthm} 
\usepackage[mathscr]{eucal}
\usepackage{color}
\usepackage{lineno}
\usepackage[abbrev, nobysame]{amsrefs}
\usepackage[T1]{fontenc}
\usepackage[dvipsnames]{xcolor}
\usepackage[pdfpagelabels,colorlinks,linkcolor=NavyBlue,citecolor=ForestGreen,urlcolor=green]{hyperref}

\usepackage{latexsym}
\usepackage{amssymb}
\usepackage{amsfonts}
\usepackage{enumerate}
\usepackage{mathrsfs}
\usepackage{color}
\usepackage{comment}
\usepackage{mathtools}

\mathtoolsset{showonlyrefs=true}

\newtheorem*{thm*}{Theorem}

\theoremstyle{plain}
\newtheorem{Theorem}{\bf Theorem}[section]
\newtheorem{Lemma}[Theorem]{\bf Lemma}
\newtheorem{Proposition}[Theorem]{\bf Proposition}
\newtheorem{Corollary}[Theorem]{\bf Corollary}
\newtheorem{Remark}[Theorem]{\bf Remark}
\newtheorem{Example}[Theorem]{\bf Example}
\newtheorem{Definition}[Theorem]{\bf Definition}

\newenvironment{theorem}{\begin{Theorem}$\!\!\!$}{\end{Theorem}}
\newenvironment{lemma}{\begin{Lemma}$\!\!\!$}{\end{Lemma}}
\newenvironment{proposition}{\begin{Proposition}$\!\!\!$}{\end{Proposition}}
\newenvironment{corollary}{\begin{Corollary}$\!\!\!$}{\end{Corollary}}
\newenvironment{remark}{\begin{Remark}$\!\!\!$}{\end{Remark}}

\newenvironment{definition}{\begin{Definition}$\!\!\!$}{\end{Definition}}

\numberwithin{equation}{section}

\newcommand{\R}{\mathbb{R}}

\newcommand{\di}{\mathop{}\!\mathrm{d}}

\def\XXint#1#2#3{{\setbox0=\hbox{$#1{#2#3}{\int}$}
\vcenter{\hbox{$#2#3$}}\kern-.5\wd0}}

\begin{document}
	\title[Blow-up rate for the H\'{e}non parabolic equation]
	{On blow-up rate for the H\'{e}non parabolic equation with Sobolev supercritical nonlinearity}
	
	
	\author{Kotaro Hisa}
	\address[K. Hisa]{Department of Applied Mathematics, Faculty of Science, Fukuoka University, 8-19-1 Nanakuma, Jonan-ku, Fukuoka-shi, Fukuoka 814-0180, Japan}
	\email{hisak@fukuoka-u.ac.jp}
    \thanks{$^\ast$Corresponding author: Kotaro Hisa.}

	\author{Yukihiro Seki}
	\address[Y. Seki]{Department of Mathematical Sciences, Graduate School of Science,Tokyo metropolitan university, 1-1 Minami-Osawa, Hachioji-shi, Tokyo 192-0397, Japan}
	\email{yseki@tmu.ac.jp}

	
	
	\date{\today}
	

	\begin{abstract}
We discuss the H\'{e}non parabolic equation
$\partial_t u = \Delta u + |x|^\sigma u^p$
in a finite ball in $\mathbb{R}^N$ under the Dirichlet boundary condition, where $N\ge1$, $p>1$, and $\sigma>0$.
We assume that the exponent $p$ is supercritical in the Sobolev sense.
Since the spatial potential term $|x|^\sigma$  vanishes at the origin, solutions seem less likely to blow up at the origin.
We construct a solution that blows up at the origin
and also carry out an analysis of blow-up rate of solutions.
In particular, if $p$ is less than the Joseph--Lundgren exponent, all blow-ups are shown to be of Type I. 
The lower bound corresponding to Type I rate is also shown for some particular blow-up solutions. 
As by products, we present a basic result on 
classification to threshold solutions for every $p>1+\sigma/N$. 
	\end{abstract}
	\maketitle



\section{Introduction}
\subsection{The H\'{e}non parabolic equation}
This paper is concerned with 
the Cauchy--Dirichlet problem for the H\'{e}non  parabolic equation
\begin{equation}
\label{eq:Henon}
\left\{
\begin{array}{ll}
	\partial_t u  = \Delta u + |x|^{\sigma}u^p, \quad  &x\in B_R,\,\,\, t\in (0,T), \vspace{3pt}\\
	u = 0, \quad  &x\in  \partial B_R,\,\,\, t\in (0,T), \vspace{3pt}\\
        u(x,0) = u_0(x), \quad  &x\in B_R, \vspace{3pt}\\
\end{array}
\right.
\end{equation}
where $p>1$, $\sigma >0$, $T>0$, and 
$u_0 \in C(\overline{B_R})$ is a nonnegative and radially symmetric  function in $\overline{B_R}$
with $u_0 = 0$ on $\partial B_R$.
Here, $ B_R := \{x\in \mathbb{R}^N; |x|<R\}$ for $N\ge3$ and $R>0$.
%
%
In the following, $u(x,t;u_0)$ and $T(u_0)$ will denote the solution of problem~\eqref{eq:Henon} with initial data $u_0$ and the maximal existence time of solution $u(x,t;u_0)$, respectively, 
but the initial data $u_0$ may be omitted if no confusion arises.
The solvability for problem~\eqref{eq:Henon} has been well-studied in recent years (see {\it e.g.},~\cites{CFL25,  CIT21, CIT22, CITT24, GW24, TW23, Yomgne22} and references cited therein).
In this paper, we study 
the blow-up phenomenon of nonnegative radially symmetric solutions of problem~\eqref{eq:Henon} and 
 assume always that all solutions $u$ blows up in  finite time $T>0$, that is,
\[
\lim_{t\nearrow T} \|u(\cdot,t)\|_{L^\infty(B_R)} = \infty.
\]
We write $B(u_0)$ for the blow-up set of the solution $u(x,t,u_0)$:
\begin{equation*}
\begin{split}
B(u_0) := &\{x \in \overline{B_R}; \exists (x_j, t_j) \in B_R \times(0,T) \,\, \mbox{such that}\\
&\qquad \qquad \qquad \qquad x_j \to x, \,\, t_j \to T \,\, \mbox{and} \,\, |u(x_j,t_j;u_0)|\to \infty\}.
\end{split}
\end{equation*}
A point in $B(u_0)$ is called a blow-up point of $u(x,t;u_0)$.
Let $p_{\rm S}(\sigma)$ and $p_{\rm JL}(\sigma)$ denote the Sobolev critical exponent and the Joseph--Lundgren exponent for the equation of \eqref{eq:Henon}, respectively, that is,
\begin{equation*}
p_{\rm S} (\sigma) := \left\{
\begin{array}{ll}
	\infty \quad  &\mbox{if} \quad N=1,2, \vspace{3pt}\\
	\displaystyle{\frac{N+2+2\sigma}{N-2}} \quad  &\mbox{if} \quad N>2, \vspace{3pt}\\
\end{array}
\right.
\end{equation*}
and
\begin{equation*}
p_{\rm JL} (\sigma) := \left\{
\begin{array}{ll}
	\infty \quad  &\mbox{if} \quad N\le 10 +4\sigma, \vspace{3pt}\\
	\displaystyle{1+ \frac{2(2+\sigma)}{N-4-\sigma-\sqrt{(2N-2+\sigma)(2+\sigma)}}} \quad  &\mbox{if} \quad N>10+4\sigma, \vspace{3pt}\\
\end{array}
\right.
\end{equation*}
for $\sigma >0$.

The blow-up phenomenon is caused by the effect of the nonlinearity $f(x,u) = |x|^\sigma u^p$ in the equation of \eqref{eq:Henon}, which depends on spatial variables.
The further away from the origin, the stronger the effect becomes, whereas at the origin, the effect disappears completely.
Therefore, it seems unlikely that the solution will blow up at the origin
and
it is natural to ask whether the origin can be a blow-up point. 
This problem has been discussed in \cites{GS11,GLS10,GS18}, in which sufficient conditions on initial data were established to ensure that the origin is not a blow-up point. 
Guo, Lin, and Shimojo \cite{GLS10} proved that the origin cannot be a blow-up point of radially symmetric solutions when  
$N\geq 3$ and $\sigma >(p-1)(N-2)/2$ (the equal case is also admitted if $N=3$).
In particular, if a solution is monotone in $t$, 
then the origin cannot be a blow-up point (\cites{GS11,GS18}).
On the other hand, 
counterexamples are known only for the special cases $N=3$, $p>5+2\sigma$ in \cite{GS11} and $N\geq 11, p>p_{\mathrm{JL}}(\sigma)$ 
in \cite{MS21}.
One of our aims is to systematically show  that counterexamples, {\it i.e.}, solutions blowing up at the origin, actually exist
for every $N\geq 3$ and $p_{\mathrm{S}}(\sigma)<p<p_{\mathrm{JL}}(\sigma)$.

%

Another aim is to clarify  the relation between blow-up at the origin and blow-up rate.
The blow-up phenomenon is divided into two types:
In this paper, a blow-up is called of Type I if there exists a constant $C>0$ such that
\begin{equation}
\label{eq:TypeI}
\|u(\cdot,t)\|_{L^\infty(B_R)} \le C (T-t)^{-\frac{2+\sigma}{2(p-1)}},
\end{equation}
for every $t\in (0,T)$;
otherwise it is called of Type II.

\begin{Remark}
Notice that our definition of Type I blow-up for problem~\eqref{eq:Henon} may be different from that of \cite{GLS10}. 
In \cite{GLS10}, a blow-up is called of Type I if there is a constant $C>0$ such that 
\begin{equation}
\label{eq:TypeI*}
\|u(\cdot,t)\|_{L^\infty(B_R)} \le C (T-t)^{-\frac{1}{p-1}},
\end{equation}
for every $t\in (0,T)$ instead of \eqref{eq:TypeI}. 
This definition is consistent with the case 
where $0$ is not a blow-up point, since 
the factor $|x|^{\sigma}$ does not degenerate (i.e., behaves like a positive constant)  
near any blow-up point when the blow-up time is approached.
This does not apply to the case where blow-up at the origin occurs. 
\end{Remark}
The study of blow-up rates of solutions is one of the most important areas in research on nonlinear differential equations and
much effort has been devoted to this field over the past decades. 
For instance, when $\sigma=0$ and $p>p_{\rm S}(0)$ with $N\geq 3$, the analysis of blow-up rate 
has been discussed in several articles. We refer to \cites{HV94,M11a,M11b,MM04,MM09,
M07,S18,S20,C17,DMW21} on this direction.
See also a monograph \cite{QS19}, which includes an extensive list of references for semilinear heat equations.
Among them, Matano and Merle \cite{MM04} showed that in the case of $\sigma =0$,
any nonnegative radially symmetric solution of problem~\eqref{eq:Henon} must exhibit Type I blow-up. 
As for the case $\sigma>0$, Phan \cite{Phan17} showed that the blow-up of any radially symmetric solution is of Type I  if $p<p_{\rm S}(\sigma)$. 
It is noteworthy that the upper blow-up rate may not be optimal 
as stated therein. In the Sobolev supercritical case $p> p_{\rm S}(\sigma)$, the only known results
on blow-up rate with $0$ being blow-up points 
are the existence of Type II blow-up solutions with exact rates due to \cite{MS21} for $p>p_{\rm JL}(\sigma)$, $N \geq 10 +4\sigma$, 
and backward self-similar (Type I) solutions obtained in \cite{FT00} for $p_{\rm S}(\sigma) < p < p_{\rm JL}(\sigma)$. 
Notice that the latter examples are, in principle, meaningful only in the whole space case.
In this paper, we show that in the case of $\sigma>0$, any nonnegative radially symmetric solution of problem~\eqref{eq:Henon} must exhibit Type I blow-up  under the assumption $p_{\rm S}(\sigma) < p < p_{\rm JL}(\sigma)$.


\subsection{Main results}
We are now ready to state our main results.
In Theorem~\ref{Theorem:A} we establish  the existence of  a nonnegative radially symmetric solution which blows up in  finite time and at the origin.

\begin{theorem}
\label{Theorem:A}
Assume $N\ge 3$, $p_{\rm S}(\sigma)<p< p_{\rm JL}(\sigma)$, and $\sigma>0$.
Then there exists a nonnegative radially symmetric function $u_0^*\in C(\overline{B_R}) \cap H^1_0(B_R)$ such that the corresponding solution $u^*(x,t)= u(x,t;u_0^*)$ of  problem~\eqref{eq:Henon} blows up in  finite time and $B(u_0^*)=\{0\}$.
\end{theorem}
%
%

In Theorem~\ref{Theorem:B}, we give a sufficient condition for nonnegative solutions not to blow up at the origin.

\begin{theorem}
\label{Theorem:B}
Assume $N\ge 1$, $p>1+\sigma/N$, and $\sigma>0$.
Suppose that a nonnegative solution $u$ of problem~\eqref{eq:Henon} blows up at $t=T\in(0,\infty)$.
If
there exists $r\in(0,R)$ such that
\begin{equation}
\label{eq:Theorem:B}
\limsup_{t \nearrow T} (T-t)^{\frac{1}{p-1}} \||\cdot|^\frac{\sigma}{p-1}u(\cdot, t)\|_{L^\infty(B_r)} =0,
\end{equation}
then $0$ is not a blow-up point.
\end{theorem}
In Theorem~\ref{Theorem:C}, we obtain lower estimates of  blow-up rate of the solution $u^*$
constructed in Theorem~\ref{Theorem:A}.
\begin{theorem}
\label{Theorem:C}
Assume  $N\ge 3$, $p_{\rm S}(\sigma)<p <p_{\rm JL}(\sigma)$, and $\sigma>0$.
Let $u^*$ be the nonnegative radially symmetric solution of problem~\eqref{eq:Henon} constructed in Theorem~{\rm \ref{Theorem:A}}, which blows up at $t=T\in (0,\infty)$.
Then there exists a constant $C>0$ such that 
\begin{equation}
\label{eq:Theorem:C1}
C(T-t)^{-\frac{2+\sigma}{2(p-1)}} \le \|u^* (\cdot,t)\|_{L^\infty(B_R)} \quad \mbox{for} \quad 0<t<T.
\end{equation}
\end{theorem}

It should be noticed that the lower estimate \eqref{eq:Theorem:C1} does not hold for general nonnegative radially symmetric solutions $u$ of problem~\eqref{eq:Henon}.
Indeed, it follows from \cite{GLS10} that if $0 \not\in B(u_0)$, then \eqref{eq:TypeI*} holds and therefore quantity 
$\|u(\cdot, t)\|_{L^\infty(B_R)}$ grows much more slowly than $ (T-t)^{-{(2+\sigma)/2(p-1)}}$ as $t\nearrow T$.
This phenomenon is naturally understood, because the potential term $|x|^\sigma$ can be essentially regarded as a constant near the blow-up point $x\neq0$.

In Theorem~\ref{Theorem:D}, we obtain upper and lower estimates of blow-up rate of any nonnegative radially symmetric solution of problem~\eqref{eq:Henon}.
Note that this theorem applies not only to solutions constructed in Theorem~\ref{Theorem:A}, but also to general nonnegative radially symmetric solutions of problem~\eqref{eq:Henon}.
\begin{theorem}
\label{Theorem:D}
Assume $N\ge 3$, $p_{\rm S}(\sigma) < p <p_{\rm JL}(\sigma)$, and  $\sigma>0$.
Let $u$ be a nonnegative radially symmetric solution of problem~\eqref{eq:Henon} which blows up at $t=T\in (0,\infty)$.
Then there exist  constants $C, C'>0$ such that 
\begin{equation}
\label{eq:Theorem:D1}
\|u(\cdot,t)\|_{L^\infty(B_R)} \le C (T-t)^{-\frac{2+\sigma}{2(p-1)}}\quad \mbox{for} \quad 0<t<T
\end{equation}
and
\begin{equation}
\label{eq:Theorem:D2}
C (T-t)^{-\frac{1}{p-1}} \le \||\cdot|^\frac{\sigma}{p-1}u(\cdot,t)\|_{L^\infty(B_R)} \le C' (T-t)^{-\frac{1}{p-1}} \quad \mbox{for} \quad 0<t<T.
\end{equation}
In particular, when $1+\sigma/N < p < p_{\rm S}(\sigma)$, \eqref{eq:Theorem:D1} and \eqref{eq:Theorem:D2} are valid for any nonnegative solution of problem~\eqref{eq:Henon}.
\end{theorem}

We note that the weighted norm $\||\cdot|^{\sigma/(p-1)}u(\cdot, t)\|_{L^\infty(B_R )}$ also blows up at the same time as that of $\| u(\cdot, t)\|_{L^\infty(B_R )}$:
\[
\lim_{t\nearrow T} \||\cdot|^\frac{\sigma}{p-1}u(\cdot, t)\|_{L^\infty(B_R )} = \infty. 
\]
See Proposition~\ref{Proposition:utoU} below. 
A recent result by Tayachi and Weissler \cite[Appendix B]{TW23} shows that 
the set of functions with this weighted norm finite plays an important role in the study of the well-posedness to the 
Cauchy problem in $\R^N$ for the H\'{e}non parabolic equation.
We are able to estimate the blow-up rate of $\||\cdot|^{\sigma/(p-1)}u(\cdot, t)\|_{L^\infty(B_R )}$ as well as 
$\|u(\cdot, t)\|_{L^\infty(B_R )}$. Moreover, the former blows up with a common rate, whereas blow-up rates of 
the latter may be different depending on whether $0$ is a blow-up point or not. See Remark~\ref{Rem:R1} below for the detail.

Combining Theorems \ref{Theorem:C} and \ref{Theorem:D}, we see that $u^*$ indeed exhibits Type I blow-up 
and that the rate is 
optimal.
We summarize this fact for completeness.

\begin{corollary}
Assume $N\ge 3$, $p_{\rm S}(\sigma) < p <p_{\rm JL}(\sigma)$, and  $\sigma>0$.
Let $u^*$ be the solution of problem~\eqref{eq:Henon} constructed in Theorem~{\rm \ref{Theorem:A}}, which blows up at $t=T\in (0,\infty)$.
Then there exist  constants $C, C'>0$ such that 
\begin{equation}
C (T-t)^{-\frac{2+\sigma}{2(p-1)}} \le \|u^*(\cdot,t)\|_{L^\infty(B_R)} \le C' (T-t)^{-\frac{2+\sigma}{2(p-1)}}\quad \mbox{for} \quad 0<t<T
\end{equation}
and
\begin{equation}
C (T-t)^{-\frac{1}{p-1}} \le \||\cdot|^\frac{\sigma}{p-1}u^*(\cdot,t)\|_{L^\infty(B_R)} \le C' (T-t)^{-\frac{1}{p-1}} \quad \mbox{for} \quad 0<t<T.
\end{equation}
\end{corollary}

\begin{remark}
\label{Rem:R1}
It is interesting to argue whether or not the estimate \eqref{eq:Theorem:C1} is also optimal for general blow-up solutions 
of problem~\eqref{eq:Henon}. Notice that it does not hold for solutions with $0$ not a blow-up point, since 
the ODE rate \eqref{eq:TypeI*} is valid for such solutions. 
Notice also that we have $\limsup_{t\nearrow T} (T-t)^{1/(p-1)} \| u(\cdot, t) \|_{L^{\infty}(B_R )} = +\infty$ for 
every radial solution of problem~\eqref{eq:Henon} with $0$ a blow-up point, where $T<\infty$ is the blow-up time. 
Indeed, this is an immediate consequence of \eqref{eq:Theorem:B} with $r>0$ small enough. 
On the other hand, the weighted norm $\||\cdot|^{\sigma/(p-1)}u(\cdot,t)\|_{L^\infty(B_R)}$ is of order $(T-t)^{-1/(p-1)}$ 
as \eqref{eq:Theorem:D2} shows, no matter whether $0$ is a blow-up point or not.
\end{remark}


Notice that the method of \cite{MM04} is inapplicable, since the translation invariance of equation, a basic 
observation of \cite{MM04}, is broken in  problem~\eqref{eq:Henon} as long as $\sigma \not= 0$.
Our arguments are inspired by those of \cite[Theorem 2]{CFG08} for
semilinear heat equations corresponding to the equation of \eqref{eq:Henon} with $\sigma=0$ and $u^p$ replaced by 
general nonlinearities $f(u)$ with polynomial growth. 
The authors of \cite{FK25} have recently extended the arguments of \cite{CFG08} further to 
semilinear heat equations with more general nonlinearities $f(u)$.
In \cite[Theorem 2]{CFG08} and \cite[Theorem 2]{FK25}, the authors assume that $u(0,t)= \max_{B_R} u(\cdot,t)$ holds for solutions $u$ of 
the Cauchy--Dirichlet problem in $B_R$ and for all $t$ sufficiently close to the blow-up time $T$. 
This condition is satisfied if the initial data $u_0$ is radially nonincreasing. 
However, in our equation of \eqref{eq:Henon}, we cannot assume that the maximum point is fixed at the origin due to the presence of
the factor $|x|^{\sigma}$ with $\sigma>0$. Nevertheless, we can prove our claim as long as $\sigma >0$ 
by means of universal estimates for radial solutions 
to the H\'{e}non parabolic equation due to \cite{Phan13}(see Proposition~{\rm \ref{Proposition:SDE}}) 
and careful intersection-comparison arguments. 

The rest of this paper is organized as follows. 
In Section~2 we prepare notation and prove some preliminary lemmas.
In Section~3 we establish a priori estimates of solutions of  problem~\eqref{eq:Henon} and  prove Theorem~\ref{Theorem:A}.
We prove Theorem \ref{Theorem:B} in Section 4.
A large part of this section is devoted to estimating blow-up rate of solutions of problem~\eqref{eq:Henon}, which leads to Theorems~\ref{Theorem:C} and \ref{Theorem:D}.
In Section~5 we classify threshold solutions of problem~\eqref{eq:Henon} for every $p>1+\sigma/N$, the ones constructed to prove Theorem~\ref{Theorem:A}.
For details of threshold solutions, see Section~3.

\section{Preliminaries}
In this section, we prepare notation and prove some preliminary lemmas.

\subsection{Singularity and decay estimate}
As stated in the previous section, the proofs of our main results are mainly carried out by improving the arguments of \cite{GS11} and \cite{CFG08}.
A key to this lies in the singular and decay estimate established by Phan~\cite{Phan13}, which generalizes  \cite[Theorem~3.1]{PQS2007} (cf.~Remark 3.4(b) therein) to the equation of \eqref{eq:Henon}.

\begin{proposition}
\label{Proposition:Phan13}
Assume $N\ge 1$,  $p>1+\sigma/N$, and $\sigma>0$.
Let $u$ be a nonnegative radially symmetric solution to the equation of \eqref{eq:Henon} in $B_R\times (0,T)$, where $T>0$.
Then there exists $C>0$ depending only on $N$, $p$, and $\sigma$ such that
\begin{equation}
\label{eq:Phan13}
\begin{split}
|x|^\frac{\sigma}{p-1} u(x,t) 
\le C
( t^{-\frac{1}{p-1}} + (T-t)^{-\frac{1}{p-1}} + |x|^{-\frac{2}{p-1}} )
\end{split}
\end{equation}
for all $0<|x|<R/2$ and $t\in (0,T)$.
\end{proposition}
\begin{proof}
See \cite[Theorem~1.3]{Phan13}.
\end{proof}

\begin{remark}
If $u$ is a global-in-time solution of problem~\eqref{eq:Henon}, 
the region of validity of 
\eqref{eq:Phan13} can be extended to $0<|x|<R$ and $t>0$. 
Moreover, the same estimate holds also for a certain blow-up solution. 
This fact plays an important role in proving Theorem \ref{Theorem:A}
(cf. Proposition~{\rm \ref{Proposition:SDE}} below).
\end{remark}

\subsection{Blow-up of weighted norm}
For $x,y \in \overline{B_R}$ and $t>0$, let $G_R=G_R(x,y,t)$ be the Dirichlet heat kernel on $B_R$.
For $x\in \mathbb{R}^N$ and $t>0$ let $G=G(x,t)$ denote the Gauss kernel, that is,
\[
G(x,t) = \frac{1}{(4\pi t)^\frac{N}{2}} \exp\left(-\frac{|x|^2}{4t}\right).
\]
We will write $e^{t\Delta}$ for the Dirichlet heat semigroup on $B_R$:
\[
[e^{t\Delta} \phi](x) := \int_{B_R} G_R(x,y,t) \phi(y) \, \di y.
\]
By the comparison principle, there holds
\[
0\le G_R(x,y,t) \le G(x-y,t)
\]
for all $x,y\in \overline{B_R}$ and $t>0$.
\begin{lemma}
\label{Lemma:CisI}
Let $u$ be a classical solution of problem~\eqref{eq:Henon}.
Then $u$ satisfies
\begin{equation*}
\begin{split}
u(x,t) 
&= [e^{t\Delta}u(\cdot,\tau)](x)+ \int_\tau^t[e^{(t-s)\Delta}|\cdot|^\sigma u(\cdot,s)^p](x)\, \di s
\end{split}
\end{equation*}
for all $x\in B_R$ and $0\le \tau<t<T$.
\end{lemma}

The following lemma is used to handle a potential term $|x|^\sigma$.
For its proof, see  \cite[Lemma~2.4]{HI21}.
\begin{lemma}
\label{Lemma:HI21}
For every $0<k<N$, there exists $C>0$ depending only on $N$ and $k$ such that, for any $x\in\mathbb{R}^N$ and $t>0$,
\[
\int_{\mathbb{R}^N} G(x-y,t) |y|^{-k} \, \di y \le C|x|^{-k}.
\]
\end{lemma}




To determine the blow-up rate of $|x|^{\sigma/(p-1)}u(x,t)$, we  show that the blow-up times of $u(x,t)$ and $|x|^{\sigma/(p-1)}u(x,t)$ coincide.

\begin{proposition}
\label{Proposition:utoU}
Assume $N\ge 1$, $p>1+\sigma/N$, and  $\sigma>0$.
Let $u$ be a nonnegative solution of problem~\eqref{eq:Henon}.
The following two claims are equivalent:
\begin{itemize}
    \item[(a)] $u(x,t)$ blows up at $t=T$;
    \item[(b)] $|x|^{\sigma/(p-1)}u(x,t)$ blows up at $t=T$.
\end{itemize}
\end{proposition}

\begin{proof}
${\rm (b)} \Rightarrow {\rm (a)}$ is obvious.
We shall show ${\rm (a)} \Rightarrow {\rm (b)}$.
Assume that $|x|^{\sigma/(p-1)}u(x,t)$ does not blow up at $t=T$, that is, $K:= \sup_{t\in (0,T)}\||\cdot|^{\sigma/(p-1)}u(\cdot,t)\|_{L^\infty(B_R)} < \infty$.
Then the blow-up point of $u(x,t)$ is only the origin. 
Consider 
\begin{equation}
\label{eq:aproximateHenon}
\left\{
\begin{array}{ll}
	\partial_t \phi  =  \Delta \phi + K^{p-1}\phi, \quad  &x\in B_R,\,\,\, t>0, \vspace{3pt}\\
	\phi = 0, \quad  &x\in  \partial B_R,\,\,\, t>0, \vspace{3pt}\\
        \phi(x,0) =  K|x|^{-\frac{\sigma}{p-1}}, \quad  &x\in B_R. \vspace{3pt}\\
\end{array}
\right.
\end{equation}
Since $p>1+\sigma/N$, we see that
\[
\phi(x,t) = e^{K^{p-1}t} [e^{t\Delta} \phi(\cdot, 0)](x) \quad \mbox{for} \quad (x,t)\in B_R\times (0,\infty)
\]
and, by the smoothing effect,
\[
\|\phi(\cdot,t)\|_{L^\infty(B_R)} \le Ct^{-\frac{N}{2}}e^{K^{p-1}t}\|\phi(\cdot,0)\|_{L^1(B_R)}<\infty \quad\mbox{for} \quad t>0.
\]
Fix $\tau \in (0,T)$.
Note that from the assumption, $u(x,\tau) \le K |x|^{-\sigma/(p-1)}$ holds for $x\in B_R$.
Since $u(x, t+\tau)$ is a subsolution of problem~\eqref{eq:aproximateHenon} in $B_R\times(0,T-\tau)$, 
by the comparison principle, 
$u(x,t+\tau)$ is bounded from above by $\phi(x,t)$ in $B_R \times (0, T-\tau)$.
This implies that $u$ is bounded at $t=T$, which is a contradiction.
Thus, assertion~(b) follows, and the proof is complete.
\end{proof}

\subsection{Steady states with Sobolev critical and supercritical nonlinearities}
In this subsection we study positive radially symmetric solutions to 
\begin{equation*}
-\Delta u = |x|^\sigma u^p \quad \mbox{in} \quad \mathbb{R}^N,
\end{equation*}
where $N\ge 3$, $p\ge p_{\rm S}(\sigma)$, and $\sigma >0$.
Positive radially symmetric classical solutions of this equation can be written in the form $u(x) = U(r)$, where $r=|x|$ and  $U\in C^2([0,\infty))$ is a positive solution of 
\begin{equation}
\label{eq:ellipticHenonradial}
U'' + \frac{N-1}{r} U' + r^\sigma U^p =0, \quad r\in(0,\infty), \qquad U'(0) = 0.
\end{equation}
Let $U_\infty(r) = c_\infty r^{-(2+\sigma)/(p-1)}$, where 
\[
c_\infty = \left[\frac{2+\sigma}{p-1} \left(N-2 - \frac{2+\sigma}{p-1} \right)\right]^\frac{1}{p-1}.
\]
Note that this is the singular steady state of equation in \eqref{eq:ellipticHenonradial}.
In this situation, solutions of problem~\eqref{eq:ellipticHenonradial} exist and satisfy the following properties.
For related results, see {\it e.g.},  \cites{Li92, QS19, Wang93}.

\begin{proposition}
\label{Proposition:SS}
Assume $N\ge 3$, $p\ge p_{\rm S}(\sigma)$, and $\sigma>0$. Given $m>0$, problem~\eqref{eq:ellipticHenonradial} possesses a unique 
positive solution $U_m\in C^2([0,\infty))$ satisfying $U_m (0) = m$.
This solution is monotone decreasing and satisfies 
\begin{equation*}
U_m(r) = m U_1(m^\frac{p-1}{2+\sigma}r ) \,\,\quad \mbox{for all} 
\,\, r>0.
\end{equation*}
If $p>p_{\rm S}(\sigma)$ then
$r^{(2+\sigma)/(p-1)} U_m (r) \to c_\infty$ as $r\to \infty$.
If $p=p_{\rm S} (\sigma)$, then
\[
U_1(r) = \left(\frac{(N+\sigma)(N-2)}{(N+\sigma)(N-2) + r^{2+\sigma}}\right)^\frac{N-2}{2+\sigma}.
\]
Let $m_1 > m_2>0$.
If $p\ge p_{\rm JL}(\sigma)$, then $U_\infty(r) > U_{m_1}(r) > U_{m_2}(r)$ for all $r>0$.
If $p_{\rm S}(\sigma) < p < p_{\rm JL}(\sigma)$, then $U_{m_1}$ and $U_{m_2}$ intersect infinitely many times, and 
$U_{m_1}$ and $U_\infty$ intersect infinitely many times as well. 
If $p=p_{\rm S}(\sigma)$, then $U_{m_1}$ and $U_{m_2}$ intersect once and $U_{m_1}$ and $U_\infty$ intersect twice.
\end{proposition}

\begin{proof}
Using the transformation:
\[
\varphi(r) = U\left(\left(\frac{2+\sigma}{2}r\right)^\frac{2}{2+\sigma}\right),
\]
problem~\eqref{eq:ellipticHenonradial} becomes
\[
\varphi'' + \frac{k-1}{r} \varphi' + \varphi^p =0, \quad r\in(0,\infty), \qquad \varphi'(0)=0,
\]
where $k= 2(N+\sigma)/(2+\sigma)$ (see {\it e.g.}, \cite[(2.2)]{Korman2020}).
Note that $k>2$ is equivalent to $N>2$ and 
${(k+2)/(k-2)} = p_{\rm S}(\sigma)$ holds.
From this point onwards, it can be shown exactly in the same way as for \cite[Theorem~9.1]{QS19}.
\end{proof}

\section{Threshold solutions}
\subsection{Definition of threshold solutions}
In this section,
we introduce the notion of threshold solution
and prove Theorem~\ref{Theorem:A}.
Most of the arguments, except for the proof of Theorem~{\rm \ref{Theorem:B}}, are similar to that of \cite[Section~3]{GS11}. 
However, some of the arguments of \cite{GS11} are restricted to $N=3$. 
We dispense with this technical assumption 
by applying Proposition \ref{Proposition:Phan13}. See Remark \ref{Remark:GS11} below for the details.

Define
\begin{equation*}
\begin{split}
& X:= \{g \in C(\overline{B_R}) \cap H^1_0(B_R);
g\ge 0 \,\, \mbox{and} \,\,g \,\, \mbox{is radially symmetric}\},\\
& A := \{u_0 \in X; u(\cdot,t;u_0) \,\, \mbox{is a global-in-time solution and} \\
&\qquad\qquad\qquad\qquad\qquad\qquad\qquad\qquad\lim_{t\to\infty} \|u(\cdot, t,u_0)\|_{L^\infty(B_R)}=0\}.
\end{split}
\end{equation*}
Let $g \in X$ and $\mu>0$,  and set $u_0 = \mu g$.
It follows from the solvability theory (see {\it e.g.}, \cite{Wang93}) that if $\mu>0$ is small enough, the solution $u(x,t; \mu g)$ exists globally-in-time
and $u(x,t; \mu g) \to 0$ as $t\to \infty$, {\it i.e.}, $\mu g\in A$.
On the other hand, it follows from Lemma~\ref{Lemma:Kaplan} below that if $\mu >0$ is large enough, the solution $u(x,t; \mu g)$ blows up in finite time.
Now, we can define 
\[
\mu^*= \mu^*(g):= \sup\{\mu>0; \mu g \in A\} \in (0,\infty).
\]
We call the solution $u(x,t; \mu^*g)$ {\it the threshold solution}.
In the following, the function $u^*(x,t)$ stands for the threshold solution $u(x,t; \mu^*g)$.

\subsection{Properties of threshold solutions}
In this subsection we collect properties of  threshold solutions.
Let $\lambda_1>0$ be the first eigenvalue of the Dirichlet Laplacian $-\Delta$ on $B_R$ and  $\phi_1>0$ be the corresponding eigenfunction such that $\|\phi_1\|_{L^1(B_R)}=1$.
First, we obtain a necessary condition for the global-in-time solvability for problem~\eqref{eq:Henon}, which is shown by the so-called Kaplan's method.
\begin{lemma}
\label{Lemma:Kaplan}
Assume $N\ge1$, $p>1+\sigma/N$, and $\sigma>0$. 
If problem~\eqref{eq:Henon} possesses a nonnegative global-in-time solution, then
\begin{equation*}
\int_{B_R} u_0(x) \phi_1(x) \, \di x \le \lambda_1^\frac{1}{p-1} \int_{B_R} |x|^{-\frac{\sigma}{p-1}} \, \di x.
\end{equation*}
\end{lemma}
\begin{proof}
Let $u$ be a global-in-time solution of problem~\eqref{eq:Henon}.
Assume that 
\begin{equation}
\label{eq:LemKaplan1}
\int_{B_R}u_0(x) \phi_1(x) \, \di x >\lambda_1^\frac{1}{p-1} \int_{B_R} |x|^{-\frac{\sigma}{p-1}} \, \di x.
\end{equation}
By H\"{o}lder's inequality and assumption $p>1+\sigma/N$, we have
\[
 \int_{B_R} |x|^\sigma u^p \phi_1 \, \di x \ge 
 \left(\int_{B_R} |x|^{-\frac{\sigma}{p-1}} \, \di x \right)^{-(p-1)} \left(\int_{B_R} u \phi_1 \, \di x\right)^p
 >0.
\]
Multiplying \eqref{eq:Henon} by $\phi_1$, integrating over $B_R$, 
and using the above inequality, we obtain
\begin{equation*}
\begin{split}
\frac{\di}{\di t} \int_{B_R} u \phi_1 \, \di x
&\ge -\lambda_1 \int_{B_R} u \phi_1 \, \di x
+ C_1 \left(\int_{B_R} u \phi_1 \, \di x\right)^p\\
\end{split}
\end{equation*}
for all $t>0$, where 
$C_1 = \left(\int_{B_R} |x|^{-\frac{\sigma}{p-1}} \, \di x \right)^{-(p-1)}.$
Set $y(t)= \int_{B_R} u(x,t) \phi_1(x) \di x$. 
The above inequality then implies that
\[
y'(t) \ge -\lambda_1y(t) + C_1y(t)^p, \quad t>0.
\]
We see from \eqref{eq:LemKaplan1} that $-\lambda_1y(0) + C_1y(0)^p>0$ and  this implies that $y(t)$ blows up in finite time,
which is a contradiction. Thus, the proof is complete.
\end{proof}

The sets $X$ and $A$ introduced above have the following property, which is proved in \cite[Lemma~3.2]{GS11} without 
restriction on $N$. 
\begin{lemma}
\label{Lemma:GS11}
The set $A$ is nonempty and relatively open in $X$.
\end{lemma}
Due to Lemma~\ref{Lemma:Kaplan}, a threshold solution $u^*$ is well-defined, {\it i.e.}, $\mu^* \in (0,\infty)$.
\begin{lemma}
Let $g\in X\setminus\{0\}$. Then $\mu^*:= \sup\{\mu>0; \mu g \in A\}\in (0,\infty)$.
\end{lemma}
\begin{proof}
Denote $u_\mu (x,t) = u(x,t;\mu g)$.
Since  $u_\mu$ is a global-in-time solution of problem~\eqref{eq:Henon} for $\mu \in (0, \mu^*)$, by Lemma~\ref{Lemma:Kaplan} we have
\[
\mu \int_{B_R} g(x) \phi_1(x) \, \di x \le 
\lambda_1^\frac{1}{p-1} \int_{B_R} |x|^{-\frac{\sigma}{p-1}} \, \di x < \infty.
\]
The left-hand side of the above inequality tends to infinity as $\mu\to \infty$ and this implies that $\mu^*< \infty$ must hold.
Thus, the proof is complete.
\end{proof}

For threshold solution $u^*$, the singularity and decay estimate \eqref{eq:Phan13} can be improved as follows, which is a key to prove Theorem~\ref{Theorem:A}.  
\begin{proposition}
\label{Proposition:SDE}
Assume $N\ge 1$, $p>1+\sigma/N$, and $\sigma>0$.
There exists $C^*>0$ depending on $\mu^*$ and $g$  such that
\begin{equation}
\label{eq:SDE}
|x|^\frac{\sigma}{p-1}u^*(x,t) \le C^*(1+t^{-\frac{1}{p-1}}+|x|^{-\frac{2}{p-1}})
\end{equation}
for all $0<|x|<R$ and $t\in(0,T^*)$, where $T^*\in(0,\infty]$ denotes the maximal existence time of $u^*$.
\end{proposition}

\begin{remark}
\label{Remark:GS11}
In \cite{GS11}  threshold solution $u^*$ is constructed under the assumption that there is no blow-up point on $\partial B_R$,  
which is shown for $N=3$. 
As a consequence, the solution $u^*$ blows up at $t=T<\infty$, but can be continued after the blow-up time $T$ 
as a global $L^1$-solution, which in particular implies that the origin must be a blow-up point. 
On the other hand, Proposition~{\rm \ref{Proposition:SDE}} implies the boundedness of threshold solutions away from the origin. 
Hence, we do not require any continuation of solutions after blow-up time. 
It is proved in Proposition~{\rm \ref{Proposition:TSblowsup}} below that the threshold solutions blow up in a finite time $T$ 
under the assumptions $N\geq 3$ and $p_{\rm S}(\sigma) <p < p_{\rm JL}(\sigma)$. 
The continuation of $u^*$ after $t=T$ then becomes completely an independent issue. It will be discussed in Appendix A. 
\end{remark}

In order to Proposition \ref{Proposition:SDE},
we introduce an energy functional.
For $w \in H^1_0(B_R)\cap L^{p+1}(B_R)$, define
\begin{equation}
 J[w] := \frac{1}{2} \int_{B_R} |\nabla w(x)|^{2} \, \di x -\frac{1}{p+1} \int_{B_R} |x|^\sigma |w(x)|^{p+1} \, \di x.
\label{energy:J}
\end{equation}
If $u$ is a classical solution of problem~\eqref{eq:Henon},
simple computations yield
\begin{equation}
\label{eq:I1}
\begin{split}
 &\frac{\di}{\di t} J[u(\cdot,t)] = 
-\int_{B_R} |\partial_t u(x,t)|^2 \, \di x \le 0, \, \\
& \frac{1}{2} \frac{\di}{\di t} \int_{B_R} u(x,t)^2 \, \di x
= -2 J[u(\cdot, t)] + \frac{p-1}{p+1} \int_{B_R} |x|^\sigma u(x,t)^{p+1} \, \di x
\end{split}
\end{equation}
for all $t\in (0,T)$.

\begin{proof}[Proof of Proposition~{\rm \ref{Proposition:SDE}}]
In what follows, $C' = C'(N,p,\sigma, R)$ denotes a generic positive constant which may vary from a line to another.
Since $u_\mu$ is a nonnegative radially symmetric global-in-time solution of problem~\eqref{eq:Henon} for $\mu \in (0,\mu^*)$,
it follows from Proposition~\ref{Proposition:Phan13} that for any $T\in (0,T^*)$,
\begin{equation}
\label{eq:A0}
\begin{split}
|x|^\frac{\sigma}{p-1} u_\mu (r,t) 
&\le C' (t^{-\frac{1}{p-1}} + (T+1 -t)^{-\frac{1}{p-1}} + |x|^{-\frac{2}{p-1}})\\
&\le C' (1+ t^{-\frac{1}{p-1}}+|x|^{-\frac{2}{p-1}})\\
\end{split}
\end{equation}
for all $0<|x|<R/2$ and $t\in (0,T)$.
Since $\mu\in (0,\mu^*)$ and $T\in (0,T^*)$ are arbitrary and $C'$ are independent of $\mu$ and $T$, letting $\mu \nearrow \mu^*$, we get \eqref{eq:SDE} for $0<|x|<R/2$ and $t\in (0,T^*)$.

We shall check that $u^*$ satisfies \eqref{eq:SDE} for all $R/4< |x| < R$ and $t\in (0,T^*)$.
To prove this, we extend the arguments of \cites{PQS2007,Phan13}.
Define
\[
A_R := \left\{x\in \mathbb{R}^N; \frac{R}{4}< |x| < R\right\}.
\]
Assume that \eqref{eq:SDE} does not hold for $u^*$. Then there exist $\mu_k \in (0,\mu^*)$, $x_k\in \{R/4\le |x|<R\}$, and $t_k\in (0,T^*)$ such that $\mu_k \nearrow \mu^*$, $t_k \nearrow T^*$, and solutions $u_k := u_{\mu_k}$ satisfy
\[
M_k := u_k(x_k, t_k) = \sup\{u_k(x,t); x\in A_R, t\in [0,t_k]\} \to \infty \quad \mbox{as} \quad k\to \infty.
\]
Note that $M_k <\infty$ for all $k=1,2,\cdots$, since $u_k(\cdot,0) = \mu_k g \in C(\overline{B_R})$.
Since $\overline{A_R}$ is compact in $\mathbb{R}^N$,  we may assume that there exists $x_0 \in \overline{A_R}$ such that $x_k \to x_0$. 
Set $r_k := |x_k|$ and $r_0 := |x_0|$.
We divide the proof into several cases.

\vspace{3pt}
\noindent
\underline{{\bf Case 1}: $r_0\in[R/4, R)$.} Let $\lambda_k := M_k^{-(p-1)/2}$, we rescale by 
\begin{equation*}
\begin{split}
&v_k (\rho,s) := \lambda_k^\frac{2}{p-1} u_k(r_k+ \lambda_k \rho, t_k + \lambda_k^2 s),\\
&\qquad\qquad\qquad\qquad \mbox{for} \quad  (\rho,s) \in Q_k :=\left(-\frac{r_k}{2\lambda_k}, \frac{R-r_k}{\lambda_k}\right) \times \left(-\frac{t_k}{\lambda_k^2}, 0\right).
\end{split}
\end{equation*}
By the definition of $v_k$, it follows that $0\le v_k \le 1 = v_k(0,0)$ and $v_k$ solves the equation
\[
\partial_s v -\partial_{\rho\rho} v - \frac{N-1}{\rho + r_k/\lambda_k}\partial_{\rho}v = |r_k + \lambda_k \rho|^\sigma v^p, \quad (\rho,s) \in Q_k.
\]
Since $v_k$ is uniformly bounded in $Q_k$, after extracting a subsequence, we may 
assume that the limit $\lim_{k\to\infty} v_k(\rho,s) =: v(\rho,s)$ exists 
and the convergence is locally uniform in $\R^N \times (-\infty, 0)$ up to 
second derivatives. 
Since $r_k \to r_0\in [R/4,R)$ and $\lambda_k \to 0$ as $k\to \infty$, we see that $v= v(\rho,s)$ satisfies
\begin{equation}
\label{eq:A1}
\partial_sv -\partial_{\rho\rho} v  = r_0^\sigma v^p, \quad (\rho,s) \in \mathbb{R} \times \left(-\infty, 0\right)
\,\,\text{ and }\,\,
v(0,0) =1.
\end{equation}
Since $\lambda_k \to 0$ as $k\to \infty$, we see from \eqref{eq:I1} that
\begin{equation}
\label{eq:A2}
\begin{split}
\int_{-\frac{t_k}{\lambda_k^2}}^0\int_{-\frac{r_k}{2\lambda_k}}^{\frac{R-r_k}{\lambda_k}} |\partial_s v_k|^2  \, \di \rho\di s
&= \lambda_k^{\frac{4}{p-1}+4} \int_{-\frac{t_k}{\lambda_k^2}}^0\int_{-\frac{r_k}{2\lambda_k}}^{\frac{R-r_k}{\lambda_k}} |\partial_t u_k(r_k + \lambda_k \rho, t_k + \lambda_k^2 s)|^2  \, \di \rho \di s\\
&=\lambda_k^{\frac{4}{p-1}+1}\int_{0}^{t_k}\int_{\frac{r_k}{2}}^{R} |\partial_t u_k(r,t)|^2  \, \di r \di t\\
&\le C'R^{1-N}\lambda_k^{\frac{4}{p-1}+1}\int_{0}^{t_k}\int_{\frac{r_k}{2}}^{R} |\partial_t u_k(r,t)|^2 r^{N-1}  \, \di r \di t\\
&\le   C'R^{1-N}\lambda_k^{\frac{4}{p-1}+1}\int_{0}^{T^*}\int_{B_R} |\partial_t u_k(x,t)|^2   \, \di x \di t\\
&= -C'R^{1-N}\lambda_k^{\frac{4}{p-1}+1}\int_{0}^{T^*} \frac{\di }{\di t}J[u_k(\cdot,t)] \, \di t\\
&= C'R^{1-N}\lambda_k^{\frac{4}{p-1}+1}[J[u_k(\cdot,0)]- J[u_k(\cdot,T^*)]]\\
& \le C'R^{1-N} (\mu^*)^2 \lambda_k^{\frac{4}{p-1}+1}\int_{B_R} |\nabla g|^2 \, \di x
\to 0 \quad \mbox{as} \quad k\to \infty.
\end{split}
\end{equation}
Since $\partial_s v_k \to \partial_s v$ uniformly in every compact subset 
of $\mathbb{R} \times \left(-\infty, 0\right)$, 
this implies that $\partial_sv\equiv0$ in
$\mathbb{R} \times \left(-\infty, 0\right)$ and $v$ is a nontrivial steady state of equation \eqref{eq:A1}.
This, however, contradicts the Liouville-type theorem for the Lane--Emden equation with $N=1$ (see {\it e.g.}, \cite{GS81}). 

\vspace{3pt}
\noindent
\underline{{\bf Case 2}: $r_0=R$.} There are  two possibilities:
\[
{\rm (i)} \quad \frac{R-r_k}{\lambda_k} \to \infty \qquad \mbox{or} \qquad {\rm (ii)} \quad\frac{R-r_k}{\lambda_k} \to c \in (0,\infty).
\]
In the case of (i), the same scaling as in Case 1 leads to a contradiction as in Case 1.
In the case of (ii), the function $v$ as in Case 1 solves the equation
\begin{equation*}
\partial_sv- \partial_{\rho\rho}v  = r_0^\sigma v^p, \quad (\rho,s) \in (-\infty,c) \times \left(-\infty, 0\right),
\end{equation*}
and
\[
v(0,0) = 1, \qquad v(c,0) = 0.
\]
A similar argument to \eqref{eq:A2} and the Liouville-type theorem in half-space for Lane--Emden equation with $N=1$ 
lead to a contradiction, whence: 
$\limsup_{k\to \infty} M_k < \infty$
and 
\[
u^*(x,t) \le C' \quad \mbox{for} \quad (x,t) \in A_R \times (0,T^*), 
\]
which yields 
\[
|x|^{\frac{\sigma}{p-1}} u^*(x,t) \le C' (1+ t^{-\frac{1}{p-1}}+ |x|^{-\frac{2}{p-1}}) \quad \mbox{for} \quad (x,t) \in A_R \times (0,T^*).
\]
Combining this estimate and \eqref{eq:A0}, we arrive at \eqref{eq:SDE}. 
The proof is complete.
\end{proof}

For threshold solution $u^*$, there are three possibilities:
\begin{itemize}
   \item[(GB)] $u^*$ exists globally-in-time and is uniformly bounded;
    \item[(G)]  $u^*$ grows up, {\it i.e.}, $\|u^*(\cdot, t)\|_{L^\infty(B_R)}< \infty$ for all $t>0$ and
    \[
    \limsup_{t\to \infty}\|u^*(\cdot, t)\|_{L^\infty(B_R)} = \infty;
    \]
    \item[(B)]  $u^*$ blows up in  finite time.
\end{itemize}
In Proposition~\ref{Proposition:TSblowsup}, we show that only case~(B) can occur in the case of $p_{\rm S}(\sigma)< p< p_{\rm JL}(\sigma)$. 
We note that the classification of possible behaviors of $u^*$ for general $p$ will be discussed in Section~5.

\begin{proposition}
\label{Proposition:TSblowsup}
Assume $N\ge3$, $p_{\rm S}(\sigma)< p< p_{\rm JL}(\sigma)$, and $\sigma>0$.
Then  threshold solution $u^*$ blows up in  finite time.
\end{proposition}

\begin{proof}
Let us first prove that case~(GB) cannot occur.
Assume that $\|u^*\|_{L^\infty(B_R)}$ is uniformly bounded with respect to $t>0$.
By the standard dynamical system argument with Lyapunov functional $J$ (cf.~\eqref{eq:I1}), 
the $\omega$-limit set $\omega(\mu^* g)$ with respect to $C^2_{\rm loc}$ topology is included in the set of nonnegative steady states.
However, when $p>p_{\rm S}(\sigma)$, there exist no positive steady states (see {\it e.g.}, \cite{Ni86}).
Thus, $\omega(\mu^* g)= \{0\}$ and  this implies that $u^* \to 0$ as $t\to\infty$, that is, $\mu^* g \in A$.
Since $A$ is an open set in $X$, we can find $\mu > \mu^*$ such that $\mu g \in A$,
which is a contradiction.
Therefore, case~(GB) cannot occur.

Let us  prove that case~(G) cannot occur.
Set $U^*(r,t) = u^*(x,t)$, where $r= |x|$.
Suppose contrary that case~(G) occurs. 
We first claim that 
\begin{equation}
\label{eq:3.11}
\mathcal{Z}_{[0,R]}(U^*(\cdot,t) - U_\infty(\cdot)) \ge 2
\end{equation}
for all $t>0$.
Otherwise, there exists $t_0 >0$ such that $\mathcal{Z}(U^*(\cdot,t_0) - U_\infty(\cdot)) $ is either $1$ or $0$.
The former case implies that there exists $r^* \in [0,R]$ such that
\[
(U^*(r^*,t_0) - U_\infty(r^*))_r = U^*(r^*,t_0) - U_\infty(r^*) = 0.
\]
By Proposition~\ref{Proposition:zeronumber}~(iii) we have
\[
\mathcal{Z}_{[0,R]}(U^*(\cdot,t) - U_\infty(\cdot)) =0 \quad \mbox{for all} \quad t>t_0.
\]
Thus, the former case can be reduced to the latter case and 
since $U^*$ is smooth for all $t>0$, we may assume that
\begin{equation*}
U^*(r,t) < U_\infty(r,t) \quad \mbox{for all} \quad r\in[0,R], \,\, t\ge t_0. 
\end{equation*}
Thus, there exists a sufficiently small $r_0 >0$ such that
\[
U^*(r,t_0) \le  U_\infty(r+r_0) \quad \mbox{for all} \quad r\in [0,R].
\]
Since $U_\infty(\cdot)$ is monotone decreasing, we obtain 
\begin{equation*}
\begin{split}
0&= \partial_t U_\infty(r+ r_0) \\
&= (U_\infty)_{rr}(r+r_0) + \frac{N}{r+r_0} (U_\infty)_r(r+r_0) + (r+r_0)^\sigma U_\infty(r+r_0)^p\\
&> (U_\infty)_{rr}(r+r_0) + \frac{N}{r} (U_\infty)_r(r+r_0) + (r+r_0)^\sigma U_\infty(r+r_0)^p.\\
\end{split}
\end{equation*}
Thus, $U_\infty(\cdot+r_0)$ is a supersolution of \eqref{eq:Henon}. 
By the comparison principle, we have
\[
U^*(r,t) \le  U_\infty(r+r_0) \quad \mbox{for all} \quad r\in [0,R],\,\, t\ge t_0.
\]
This implies that $U^*(r,t)$ is uniformly bounded in $[0,R]\times (0,\infty)$. 
On the other hand, the supercriticality $p>p_{\rm S}(\sigma)$ implies that $U\equiv 0$ is the only nonnegative steady state.
Therefore, $\lim_{t\to\infty} u^*(x,t)=0$ holds for any $x\in B_R$.
This contradicts Lemma~\ref{Lemma:GS11}, and the claim \eqref{eq:3.11} is established.

Next, it is known from \cite{FT00} that, when $p_{\rm S}(\sigma) < p < p_{\rm JL}(\sigma)$, there exists a backward self-similar solution $\Phi_m$ of \eqref{eq:Henon}:
\begin{equation}
\label{eq:BSSS}
\Phi_m (r,t) := (T-t)^{-\frac{2+\sigma}{2(p-1)}} \varphi_m\left(\frac{r}{\sqrt{T-t}}\right), \quad t<T,
\end{equation}
for any $T>0$, where the profile function $\varphi_m$ is positive on $[0,\infty)$ and satisfies 
\begin{equation*}
\left\{
\begin{array}{ll}
    \displaystyle{\varphi_m'' + \left(\frac{N-1}{r} - \frac{r}{2}\right) \varphi_m' -\frac{\beta}{2} \varphi_m +r^\sigma \varphi_m^p  =0}, \quad  &r>0, \vspace{3pt}\\
    \varphi_m'(0)=0, \quad  & \vspace{3pt}\\
    \displaystyle{\lim_{r\to\infty} r^\frac{2+\sigma}{p-1} \varphi_m(r) =m}, \quad  & \vspace{3pt}\\
\end{array}
\right.
\end{equation*}
for a certain $m \in (0,c^*)$.

Let us consider the intersection $U^*$ and $\Phi_m$.
Fix $t_1 \in (0,1)$.
Since $U^*_r(R,t_1)<0$, there exists a positive constant $\varepsilon \in (0,R/2)$ such that
$U^*_r(r,t_1) <0$ for all $r \in [R-\varepsilon,R]$. 
Set 
\[
\gamma_1 := \min_{r \in [0,R-\varepsilon]} \frac{U^*(r,t_1)}{2}, \quad \gamma_2 := \max_{r \in [R-\varepsilon,R]} U^*_r(r,t_1).
\]
Clearly, $\gamma_1 >0$ and $\gamma_2 <0$.
By \eqref{eq:BSSS}, we can choose $T>0$ sufficiently large so that $\Phi_m(r,t_1)<\gamma_1$ for all $r>0$ and $(\Phi_m)_r (r,t_1) > \gamma_2$ for all $r \in [R-\varepsilon,R]$.
On the other hand, from the boundary condition, we have $U^*(R,t) =0$ and $\Phi_m(R,t)>0$ for all $t>0$. This implies that $\mathcal{Z}_{[0,R]}(U^*(\cdot,t_1) - \Phi_m(\cdot, t_1)) =1$.
Combining this with Proposition~\ref{Proposition:zeronumber}, we have 
\begin{equation}
\label{eq:3.13}
\mathcal{Z}_{[0,R]}(U^*(\cdot,t) - \Phi_m(\cdot,t)) \le 1, \quad t\ge t_1.
\end{equation}
Note that $\Phi_m(r,T) = m r^{-(2+\sigma)/(p-1)}$ for all $r>0$.
Recall $U_\infty(r) = c_\infty r^{-(2+\sigma)/(p-1)}$ and \eqref{eq:3.11}.
Let $r_1$ and  $r_2$ be the smallest and largest intersection points of $U^*$ and $U_\infty$ in $(0,R)$.
Choose $t_2 < T$ sufficiently close to $T$. 
Then there must be at least one intersection of $U^*$ and $\Phi_m$ in $(0,r_1)$ and $(r_2, R)$, respectively.
This implies that $\mathcal{Z}(U^*(\cdot, t_2)- \Phi_m(\cdot, t_2)) \ge 2$, which contradicts \eqref{eq:3.13} and Proposition~\ref{Proposition:zeronumber}~(ii).
Thus, case~(G) does not occur and we conclude that $u^*$ blows up in finite time.

Finally, we shall show that  threshold solution $u^*$ blows up at the origin.
Since $\mu^* g \in L^\infty(B_R)$, problem~\eqref{eq:Henon} with $u_0 = \mu^* g$ is locally solvable.
This  implies that
$u^*$ blows up in  finite time $T^* \in (0, \infty)$ and the proof is complete.
\end{proof}

We are now ready to prove Theorem~\ref{Theorem:A}. 
\begin{proof}[Proof of Theorem~{\rm \ref{Theorem:A}}]
Let  $g \in C(\overline{B_R}) \cap H^1_0(B_R)$ be a nonnegative radially symmetric function with $g\not \equiv0$.
It follows from Proposition~\ref{Proposition:TSblowsup} that the threshold solution $u^*$ blows up in  finite time.
By Proposition~\ref{Proposition:SDE} we know that the only blow-up point of $u^*$ is the origin.
Thus, the proof is complete.
\end{proof}

\section{Nondegeneracy of blow-up point and estimates of blow-up rate}
\subsection{Nondegeneracy of blow-up point}
In this subsection, we prove Theorem~\ref{Theorem:B}.
The proof is based on the arguments in \cite{GSU07}. 
\begin{proof}[Proof of Theorem~{\rm \ref{Theorem:B}}]
We can assume, without loss of generality, that
$R>0$ is sufficiently large and $r=1$.
It follows from \eqref{eq:Theorem:B} that there exist a sufficiently small $\delta>0$ and $\tau \in (0,T)$ sufficiently close to $T$ such that
\begin{equation*}
\mathcal{U}(x,t) = |x|^\frac{\sigma}{p-1} u(x,t) \le \delta (T-t)^{-\frac{1}{p-1}} \quad \mbox{for} \quad (x,t) \in B_1\times(\tau, T). 
\end{equation*}
Let $\phi \in C^\infty(B_1)$ be a cut-off function satisfying $0\le \phi \le 1$ and $\phi\equiv 1$ in $B_{1/2}$.
The function $w(x,t) := \phi(x) u(x,t)$ satisfies the equation 
\[
\partial_t w = \Delta w -g(x,t) + \phi(x) |x|^\sigma u(x,t)^p \quad \mbox{in} \quad B_1 \times (0,T).
\]
Let $\tau \in(0,T)$ be close enough to $T$.
Then the representation formula gives us 
\[
w(\cdot, t) = e^{(t-\tau)\Delta} w(\cdot,\tau) + \int_\tau^t e^{(t-s) \Delta} (-g(\cdot,s) + \phi |\cdot|^\sigma u(\cdot,s)^p) \, \di s, 
\]
where $\{e^{t\Delta}\}_{t\ge0}$ is the heat semigroup in $B_1$ with the Dirichlet boundary condition.
By the estimate of the derivative of the Dirichlet heat kernel $G_1$ on $B_1$ (see \cite[Theorem~2.1]{Zhang06}), the comparison principle, and Lemma~\ref{Lemma:HI21} we have 
\begin{equation*}
\begin{split}
&\int_\tau^t|[e^{(t-s)\Delta} g(\cdot, s)](x)| \, \di s\\
\le & 2 \int_\tau^t |\nabla e^{(t-s)\Delta} ( u(\cdot,s)\nabla\phi)| \, \di s
+
\int_\tau^t | e^{(t-s)\Delta} (u(\cdot,s)\Delta \phi)| \, \di s\\
\le & C\int_\tau^t (t-s)^{-\frac{1}{2}}\|\mathcal{U}(\cdot,s)\|_{L^\infty(B_1)} \int_{\mathbb{R}^N}G(x-y,t-s)|y|^{-\frac{\sigma}{p-1}} \, \di y \di s\\
& \quad + \int_\tau^t \|\mathcal{U}(\cdot,s)\|_{L^\infty(B_1)} \int_{\mathbb{R}^N}G(x-y,t-s) |y|^{-\frac{\sigma}{p-1}} \,  \di y\di s\\
\le & C |x|^{-\frac{\sigma}{p-1}} \int_{\tau}^t 
(t-s)^{-\frac{1}{2}}\|\mathcal{U}(\cdot,s)\|_{L^\infty(B_1)}  \, \di s
\end{split}
\end{equation*}
for all $x\in B_1$ and $t\in (\tau,T)$.
Set $\mathcal{W}(x,t) = |x|^{\sigma/(p-1)} w(x,t)$.
Similarly, we obtain
\begin{multline}
w(x,t) 
\le C |x|^{-\frac{\sigma}{p-1}} \|\mathcal{W}(\cdot,\tau)\|_{L^\infty(B(0,1))}
\\
+ C|x|^{-\frac{\sigma}{p-1}}\int_\tau^t (t-s)^{-\frac{1}{2}}
\|\mathcal{U}(\cdot,s)\|_{L^\infty(B_1)} \, \di s\\
+ C |x|^{-\frac{\sigma}{p-1}}\int_\tau^t 
\|\mathcal{W}(\cdot,s)\|_{L^\infty(B_1)}\|\mathcal{U}(\cdot,s)\|_{L^\infty(B_1)}^{p-1}\, \di s
\end{multline}
for all $x\in B_1$ and $t\in (\tau, T)$. 
Multiplying the last inequality by $|x|^{\frac{\sigma}{p-1}}$ and taking $L^{\infty}$ norm in $B_1$, we arrive at 
\begin{equation*}
\begin{split}
\|\mathcal{W}(\cdot,t)\|_{L^\infty(B_1)}
&\le  C_1  \|\mathcal{W}(\cdot,\tau)\|_{L^\infty(B_1)}+ C_2\int_\tau^t (t-s)^{-\frac{1}{2}} \|\mathcal{U}(\cdot,s)\|_{L^\infty(B_1)} \, \di s\\
& \quad + C_3 \int_\tau^t \|\mathcal{W}(\cdot,s)\|_{L^\infty(B_1)}\|\mathcal{U}(\cdot,s)\|_{L^\infty(B_1)}^{p-1}\, \di s\\ 
\end{split}
\end{equation*}
for all $t\in (\tau,T)$. 
By the assumption \eqref{eq:Theorem:C1}, for every $\varepsilon\in(0,1)$ there exists $s_0 > 0$ large enough such that 
$\|\mathcal{U}(\cdot,s)\|_{L^\infty(B_1)}^{p-1} < \varepsilon^{p-1} /(T-s)$ for $s> s_0$. 
Applying this estimate in the above inequality and performing a Gronwall-type argument, we then obtain 
\[
\|\mathcal{W}(\cdot,t)\|_{L^\infty(B_{1/2})} \le C (T-t)^{-\frac{\varepsilon}{p-1}} \quad \mbox{for} \quad  s_0<t<T.
\]
With this improved estimate, we repeat the previous procedure finitely many times, 
to get the boundedness of $|x|^{\sigma/(p-1)} u(x,t)$.
This implies that $0$ is not a blow-up point.
For more details, see \cite[Lemma~4.2.1]{GSU07}.
The proof is complete.
\end{proof}

\subsection{Lower estimates without restrictions on $p$}
In this section we shall study  blow-up rate of solutions $u$ of problem~\eqref{eq:Henon}. 
We define
\[
M(t) := \|u(\cdot, t)\|_{L^\infty(B_R)}  \in[0, \infty) \quad \mbox{for} \quad t\in(0,T).
\] 
Since $\overline{B_R}$ is compact in $\mathbb{R}^N$, 
for each $t\in (0,T)$ there is a point $x(t) \in {B_R}$ such that 
\begin{equation}
\label{eq:defofxt}
M(t) = \|u(\cdot, t)\|_{L^\infty(B_R)}  = u (x(t),t).
\end{equation}
If more than one such point is found, the one with the lowest absolute value is taken.
 It is readily seen that function $t \mapsto M(t)$ is Lipschitz continuous in $(0, T)$ and $\lim_{t\nearrow T}M(t)=\infty$ holds
 for every finite-time blow-up solution.

First, we obtain a lower estimate of  blow-up rate without restrictions on $p$.
\begin{proposition}
\label{Proposition:lowerU}
Assume $N\ge 1$, $p>1+\sigma/N$, and $\sigma>0$.
Let $u$ be a nonnegative  solution of problem~\eqref{eq:Henon} which blows up at $t=T\in (0,\infty)$.
Then there exists a constant $C>0$ such that 
\begin{equation*}
C(T-t)^{-\frac{1}{p-1}} \le \||\cdot|^\frac{\sigma}{p-1} u(\cdot, t)\|_{L^\infty(B_R)} \quad \mbox{for} \quad 0<t<T.
\end{equation*}
\end{proposition}

\begin{proof}
We set 
\[
\mathcal{U}(x,t) = |x|^\frac{\sigma}{p-1}u(x,t) \quad \mbox{and} \quad \mathcal{M}(t) = \|\mathcal{U}(\cdot, t)\|_{L^\infty(B_R)}.
\]
It follows from Lemmas~\ref{Lemma:CisI}, \ref{Lemma:HI21}, and the comparison principle that
\begin{equation*}
 \begin{split}
u(x,t) 
&= \int_{B_R} G_R(x,y,t-\tau) |y|^{-\frac{\sigma}{p-1}} |y|^{\frac{\sigma}{p-1}}u(y,\tau) \, \di y\\
&\qquad\qquad+ 
\int_\tau^t \int_{B_R} G_R(x,y,t-s)|y|^{-\frac{\sigma}{p-1}} (|y|^\frac{\sigma}{p-1} u(y,s))^p \, \di y \di s\\
&\le \mathcal{M}(\tau) \int_{\mathbb{R}^N} G(x-y,t-\tau) |y|^{-\frac{\sigma}{p-1}} \, \di y\\
&\qquad\qquad+ \int_\tau^t\mathcal{M}(s)^p
\int_{\mathbb{R}^N} G(x-y,t-s) |y|^{-\frac{\sigma}{p-1}} \, \di y \di s\\
&\le C_1 |x|^{-\frac{\sigma}{p-1}} \mathcal{M}(\tau)
+
C_2 |x|^{-\frac{\sigma}{p-1}}\int_\tau^t\mathcal{M}(s)^p \, \di s
 \end{split}   
\end{equation*}
for all $x\in B_R$ and $0<\tau<t<T$, where $C_1, C_2>0$ are constants.
Since $x\in B_R$ is arbitrary, we have
\[
\mathcal{M}(t) 
\le C_1  \mathcal{M}(\tau)
+
C_2\int_\tau^t\mathcal{M}(s)^p \, \di s
\]
for all $0<\tau < t <T$.

Fix $\tau\in(0,T)$ arbitrary.
Thanks to Proposition~\ref{Proposition:utoU}, we can 
choose
\[
t=t_0 := \min\{t' \in (\tau,T); \mathcal{M}(t') = 2C_1\mathcal{M}(\tau)\} \in (\tau, T).
\]
Then we obtain
\begin{equation*}
\begin{split}
2C_1 \mathcal{M}(\tau) 
&\le C_1 \mathcal{M}(\tau) + (2C_1)^pC_2(t_0-\tau)\mathcal{M}(\tau)^p 
\end{split}
\end{equation*}
for all $0<\tau<T$, which implies 
\[
C(T-t)^{-\frac{1}{p-1}} \le \mathcal{M}(t)
\]
for all $t\in(0,T)$.
Thus, the proof is complete.
\end{proof}

Theorem~\ref{Theorem:C} immediately follows from Proposition~\ref{Proposition:lowerU}.

\begin{proof}[Proof of Theorem~{\rm \ref{Theorem:C}}]
Let $u^*$ be a threshold solution of problem~\eqref{eq:Henon} in $B_R\times (0,T)$.
For each $t\in (T/2,T)$ we can find $X(t)\in B_R$ satisfying 
\[
\mathcal{M}(t) = |X(t)|^\frac{\sigma}{p-1} u^*(X(t),t).
\]
It follows from Proposition~\ref{Proposition:SDE} that
\[
|X(t)|^\frac{2}{p-1} \mathcal{M}(t) = |X(t)|^\frac{2+\sigma}{p-1} u^*(X(t),t) \le C.
\]
Then by Proposition~\ref{Proposition:SDE} we have
\[
|X(t)|\le C \mathcal{M}(t)^{-\frac{p-1}{2}} \le C\sqrt{T-t}.
\]
Summarizing the calculations so far, we obtain
\begin{equation*}
\begin{split}
C(T-t)^{-\frac{1}{p-1}} \le \mathcal{M}(t)
&= |X(t)|^\frac{\sigma}{p-1} u^*(X(t),t)
\le C (T-t)^\frac{\sigma}{2(p-1)} M(t)
\end{split}
\end{equation*}
for all $t\in(T/2,T)$. This inequality implies \eqref{eq:Theorem:C1}.
For the case of $0<t\le T/2$,
\eqref{eq:Theorem:C1} is immediately obtained since $u_0$ is bounded.
Thus, 
the proof is complete.
\end{proof}

We notice that we can prove the following proposition using exactly the same arguments as in the proof of Theorem~\ref{Theorem:C}.

\begin{proposition}
\label{Proposition:BRforThreshold}
Assume $N\ge 1$, $p>1+\sigma/N$, and $\sigma>0$.
If  threshold solution $u^*$ blows up at $t=T<\infty$.
Then \eqref{eq:Theorem:C1} 
is  also valid.
\end{proposition}

\subsection{Intersection numbers}
In order to establish upper estimates of  blow-up rate,
we introduce the zero number of $f\in C([0,R])$ and collect its properties.
Define 
\begin{equation*}
\begin{split}
\mathcal{Z}_{[0,R]} (f) := \sup \{k \in \mathbb{N}; \,\,
&\mbox{there exists} \, \,
 0< x_0 < x_1 < \cdots < x_k <R\\
&\mbox{such that} \,\, f(x_i)f(x_{i+1})<0 \,\,\mbox{for} \,\, 0\le i <k \}
\end{split}
\end{equation*}
as the zero number of $f$ on $[0,R]$.
The following lemma gives properties of $\mathcal{Z}_{[0,R]}(f)$. For details, see \cite[Lemma~3.5]{GS11}
(see also \cite[Theorem~52.28]{QS19}).
\begin{proposition}
\label{Proposition:zeronumber}
Let $W(r,t)$ be a non-zero smooth radially symmetric solution of 
\begin{equation*}
\left\{
\begin{array}{ll}
	\partial_t w  = \Delta w + a(|x|,t)w, \quad  &x\in B_R,\,\,\, t\in (t_1,t_2), \vspace{3pt}\\
	w \neq 0, \quad  &x\in  \partial B_R,\,\,\, t\in (t_1,t_2), \vspace{3pt}\\
\end{array}
\right.
\end{equation*}
where $a(r,t)$ is continuous on $[0,R)\times(t_1,t_2)$.
Then
\begin{itemize}
\item[(i)] $\mathcal{Z}_{[0,R]}(W (\cdot,t)) < \infty$ for all $t\in (t_1,t_2)$;
\item[(ii)] the function $t \mapsto \mathcal{Z}_{[0,R]}(W(\cdot,t))$ is nonincreasing;
\item[(iii)] if $W(r_0, t_0) = W_r(r_0,t_0) =0$ for some $r_0 \in [0,R]$ and  $t_0 \in (t_1,t_2)$,
then 
$\mathcal{Z}_{[0,R]}(W(\cdot, t)) > \mathcal{Z}_{[0,R]}(W(\cdot, s))$ for all $t_1<t<t_0<s<t_2$.
\end{itemize}
\end{proposition}
By Proposition~\ref{Proposition:zeronumber}, we obtain the following lemma.
\begin{lemma}
\label{Lemma:4.1}
Let $u$ be a nonnegative radially symmetric solution of problem~\eqref{eq:Henon} which  blows up in at $t= T\in (0,\infty)$.
Then there exists $t_0 \in (0,T)$ such that 
\[
u_t(0, t)\,\, \mbox{does not change its sign}
\quad \mbox{and} \quad 
\mathcal{N}(t):=\mathcal{Z}_{[0,R]} [u_t(\cdot, t)]\,\, \mbox{is a constant},
\]
for $t\in(t_0,T)$.
\end{lemma}

\begin{proof}
See \cite[Lemma~23.12]{QS19} for the case of $\sigma=0$.
We note that the function $u_t$ is a radially symmetric classical solution of 
\begin{equation*}
\left\{
\begin{array}{ll}
	\partial_t V  = \Delta V + p|x|^{\sigma}u^{p-1}V, \quad  &x\in B_R,\,\,\, t\in (0,T), \vspace{3pt}\\
	V = 0, \quad  &x\in  \partial B_R,\,\,\, t\in (0,T). \vspace{3pt}\\
\end{array}
\right.
\end{equation*}
By Proposition~\ref{Proposition:zeronumber}, $\mathcal{N}(t)$ is finite and nonincreasing, hence constant on $(t_0,T)$ for some $t_0\in(0,T)$.
Moreover, by the symmetry of the solution, 
if $u_t(0,t)=0$, then $u_t(\cdot, t)$ has a degenerate zero at $r=0$, so that the function $\mathcal{N}$ drops at time $t$.
Consequently, $u_t(0,t)$ does not change its sign on $(t_0,T)$.
\end{proof}

In the rest of Section~4, we shall prove Theorem~\ref{Theorem:D}.
If $0$ is not a blow-up point, Theorem~\ref{Theorem:D} can be proved immediately by \cite[Theorem~5]{GLS10}.
In what follows, we may assume that
\[
|x(t)| < \frac{R}{2} \quad \mbox{for}\quad 0<t<T,
\]
where $x(t)$ is as in \eqref{eq:defofxt}.
Proposition~\ref{Proposition:rescale} is a key to the proof of Theorem~\ref{Theorem:D}.
\begin{proposition}
\label{Proposition:rescale}
Assume $N\ge 3$, and $p> p_{\rm S}(\sigma)$,  $\sigma>0$. Let $u$ be a solution of problem~\eqref{eq:Henon} which blows up at $t=T\in (0,\infty)$.
Suppose that the blow-up is of Type II, that is,
\begin{equation}
\label{eq:4.1}
\limsup_{t \nearrow  T} (T-t)^\frac{2+\sigma}{2(p-1)} \|u(\cdot,t)\|_{L^\infty(B_R)} = \infty.
\end{equation}
Then there exists a sequence $\{t_j\}_{j=1}^\infty\subset (T/2,T)$ with $t_j\nearrow  T $  such that 
\begin{equation}
\label{eq:4.2}
\lim_{j\to \infty} \frac{u_t(0,t_j)}{M(t_j)^{\frac{2p+\sigma}{2+\sigma}} }=0
\end{equation}
holds. 
Moreover, the sequence of the functions
\begin{equation}
\label{eq:4.3}
v_j(r,s) := \frac{1}{M(t_j)} u( \lambda_j r, t_j + \lambda_j^2s) \quad \mbox{with} \quad \lambda_j := M(t_j)^{-\frac{p-1}{2+\sigma}}
\end{equation}
remains locally bounded in $C^{2,1} ([0,R)\times \mathbb{R})$ and converge, along some subsequence, to a radially symmetric steady state $U_1$ if $s=0$.
\end{proposition}


In order to prove Proposition~\ref{Proposition:rescale}, 
we invoke blow-up arguments originally given by \cite{CFG08} for equations of the form $u_t = \Delta u + f(u)$ 
and that have recently been generalized in \cite{FK25}. 
An essential difficulty arises in applying this technique to our problem~\eqref{eq:Henon} due to 
the presence of the degenerate potential $|x|^{\sigma}$ with $\sigma>0$, 
since radial solutions $u(r,t)$ may not be nonincreasing in $r$,
even if so they are at $t=0$.
In \cite[Theorem 2]{CFG08}, the authors showed an assertion analogous to Proposition~\ref{Proposition:rescale} 
under the assumption that $u(0,t)= \max_{\Omega} u(\cdot,t)$ holds. 
Moreover, they conjecture that their proof can be carried out, even without this assumption, as long as the conclusions of Lemmas~\ref{Lemma:4.2} and \ref{Lemma:4.4} are true (see \cite[(1.5)]{CFG08} and subsequent paragraphs). 
The proof of Proposition~\ref{Proposition:rescale} shows that the conjecture is correct 
for our problem~\eqref{eq:Henon} if $\sigma>0$.

We prepare a few lemmas.
Lemma~\ref{Lemma:4.2} claims that when the solution $u$ exhibits Type II blow-up, the asymptotic behaviors of the maximum value of $u$ and the value of $u$ at the origin are almost identical near the blow-up time.

\begin{lemma}
\label{Lemma:4.2}
Let $u$ be a solution of problem~\eqref{eq:Henon} in $B_R\times(0,T)$, where $T\in (0,\infty)$
and
$\{t_j\}_{j=1}^\infty \subset (T/2,T)$ be a sequence such that  $t_j \nearrow T$ as $j\to \infty$.
If 
\begin{equation}
\label{eq:4.4}
 \lim_{j\to \infty} \frac{u(0,t_j)}{M(t_j)} =0,
\end{equation}
then one has
\[
\lim_{j\to \infty} (T-t_j)^{\frac{2+\sigma}{2(p-1)}} M(t_j) =0.
\]
\end{lemma}
\begin{proof}
We may assume that there exists $j_0 \in \mathbb{N}$ such that 
\begin{equation}
\label{eq:addtionalassumption}
M(t_j) = \max_{t \in [0,t_j]} M(t) \quad \mbox{for every } j \geq j_0 .
\end{equation}
Indeed, if not, 
we can construct a new sequence $\{t_j'\}_{j=1}^\infty \subset (T/2, T)$ with $0\leq t_j' < t_j$, $t_j'< t_{j+1}'$ $(j=1,2,...)$ 
such that 
\begin{equation}
M(t_j') = \max_{t \in [0,t_j]} M(t)
\ge  M(t_j) \quad \mbox{for every } j =1,2,...
\label{ineq:S4-1}
\end{equation}
Since $M(t_j ) \to \infty$ as $j\to\infty$, it follows from \eqref{ineq:S4-1} that $t_j' \to T$ as $j\to \infty$. 
Due to Lemma~\ref{Lemma:4.1}, $u_t(0,t)$ does not change its sign for $t \in (t_0 ,T)$. 
We may assume $t_j \in (t_0 , T)$ for every $j$. 
Suppose $u_t(0,t)>0$ for $t \in (t_0 ,T)$.
Then $u(0,t_j') < u(0,t_j)$, whence by \eqref{ineq:S4-1}, 
\[
0\le \frac{u(0,t_j')}{M(t_j')} < \frac{u(0,t_j)}{M(t_j)} \to 0 \quad \mbox{as}\quad j\to \infty. 
\]
Next, suppose $u_t(0,t)<0$ for $t \in (t_0 ,T)$. Then
we see that \eqref{eq:4.4} holds with $t_j$ replaced by $t_j'$ since $u(0,t)$ is uniformly bounded in $(t_0 , T)$.
In the both cases, we have shown that \eqref{eq:4.4} holds with $t_j$ replaced by $t_j'$. 
Moreover, we have
\begin{equation}
\limsup_{j\to \infty}(T-t_j)^{\frac{2+\sigma}{2(p-1)}} M(t_j) \le \limsup_{j\to \infty}(T-t_j')^{\frac{2+\sigma}{2(p-1)}} M(t_j').
\label{ineq:S4-2}
\end{equation}
due to \eqref{ineq:S4-1}.
Therefore, it suffices to replace $\{t_j\}_{j=1}^\infty$ with $\{t_j'\}_{j=1}^\infty$, which satisfies \eqref{eq:addtionalassumption}. 
Once we obtain $\limsup_{j\to \infty} (T-t_j')^{{(2+\sigma)}/{2(p-1)}} M(t_j') <\infty$, the conclusion of Lemma \ref{Lemma:4.2}
readily follows from \eqref{ineq:S4-2}. 

First, we assume that 
\[
\liminf_{j\to \infty} \frac{r_j}{\lambda_j} <\infty,
\]
where $\lambda_j$ is as in \eqref{eq:4.3} and $r_j := |x(t_j)|$.
We may assume that there exists $P\in[0,\infty)$ such that $\lambda_j^{-1}r_j \to P$ as $j\to \infty$.
Define
\begin{equation}
\tilde{v}_j(r,s) := \frac{1}{M(t_j)} u(r_j +  \lambda_j r, t_j + \lambda_j^2s). 
\end{equation}
We see that $0\le \tilde{v}_j \le 1$
in $(-\lambda_j^{-1} r_j, \lambda_j^{-1}(R-r_j))\times (-\lambda_j^{-2}t_j,0)$ and
 $\tilde{v}_j$ satisfies 
\begin{multline*}
	\displaystyle{\partial_s \tilde{v}_j = \partial_r^2 \tilde{v}_j + \frac{N-1}{\lambda^{-1}_jr_j+r} \partial_r \tilde{v}_j + |\lambda^{-1}_jr_j+r|^\sigma \tilde{v}_j^p}\\ 
    \text{ for }\,\,
r\in \left(-\frac{r_j}{\lambda_j}, \frac{R-r_j}{\lambda_j}\right),\,\, s\in \left(-\frac{t_j}{\lambda_j^2}, 0\right)
\end{multline*}
and
\[
\partial_r \tilde{v}_j\left(-\frac{r_j}{\lambda_j},s \right)=0 \quad \mbox{for}\quad s\in \left(-\frac{t_j}{\lambda_j^2}, 0\right).
\]
By \eqref{eq:4.4}  we have
\[
\tilde{v}_j\left(-\frac{r_j}{\lambda_j},0\right) = \frac{u(0,t_j)}{M(t_j)}  \to 0 \quad \mbox{as} \quad j\to \infty.
\]
Let $\tilde{D}= (-P,\infty)\times (-\infty,0)$.
Therefore, by interior parabolic estimates, it follows that after extracting a subsequence, still denoted by $\{t_j\}$,
$\tilde{v}_j$ converges in $C_{\rm loc}^{2+\alpha, 1+\alpha/2}(\tilde{D})$ to a nonnegative solution $\tilde{v}_\infty$ of
\begin{equation*}
\left\{
\begin{array}{ll}
	\displaystyle{\partial_s \tilde{v}_\infty = \partial_r^2 \tilde{v}_\infty + \frac{N-1}{P+r} \partial_r \tilde{v}_\infty +  |P+r|^\sigma v_\infty^p},\quad &  (r,s)\in \tilde{D}, \vspace{3pt}\\
   \partial_r \tilde{v}_\infty(-P,s)  =0, \quad & s\in (-\infty,0), \vspace{3pt}\\
    \tilde{v}_\infty(0,0)=1, \,\, \tilde{v}_\infty(-P,0) =0.\\
\end{array}
\right.
\end{equation*}
However, since $\tilde{v}_\infty(-P,0)=0$ and $\tilde{v}_\infty$ is nonnegative, by the strong maximum principle, there does not exist such a solution.
Therefore, 
\begin{equation}
\label{eq:4.5}
\lim_{j\to \infty} \frac{r_j}{\lambda_j} =\infty
\end{equation}
must hold.

It follows from Proposition~\ref{Proposition:Phan13} that
\begin{equation*}
\begin{split}
\frac{r_j}{\lambda_j} 
&= \left(r_j^\frac{2+\sigma}{p-1}u(r_j, t_j)\right)^\frac{p-1}{2+\sigma}
\le C\left(1+r_j^\frac{2}{p-1} + \left(\frac{r_j}{\sqrt{T-t}}\right)^\frac{2}{p-1}\right)^\frac{p-1}{2+\sigma} 
\end{split}
\end{equation*}
By \eqref{eq:4.5} we see that
\begin{equation}
\label{eq:r_j}
\lim_{j\to\infty} \frac{r_j}{\sqrt{T-t_j}} =\infty.
\end{equation}
Using this and Proposition~\ref{Proposition:Phan13} again, we obtain
\begin{equation*}
\begin{split}
(T-t_j)^{\frac{2+\sigma}{2(p-1)}} M(t_j) 
\le C \left( \frac{r_j}{\sqrt{T-t_j}} \right)^{-\frac{\sigma}{p-1}}
\left(1 + \left( \frac{r_j}{\sqrt{T-t_j}} \right)^{-\frac{2}{p-1}} \right)  
\end{split}
\end{equation*}
for all $j$.
Thus, the desired conclusion follows from \eqref{eq:r_j}.
\end{proof}

\begin{lemma}
\label{Lemma:4.3}
Assume $p>p_{\rm S}(\sigma)$ and let $u$ be a solution of problem~\eqref{eq:Henon} in $B_R\times(0,T)$, where $T\in (0,\infty)$.
If $u$ exhibits Type II blow-up, then one has 
\begin{equation}
\label{eq:4.6}
\liminf_{t\nearrow T} (T-t)^\frac{2+\sigma}{2(p-1)} \|u(\cdot, t)\|_{L^\infty(B_R)}>0.
\end{equation}
\end{lemma}

\begin{proof}
Suppose that
\begin{equation}
\label{eq:4.7}
\liminf_{t\nearrow T} (T-t)^\frac{2+\sigma}{2(p-1)} \|u(\cdot, t)\|_{L^\infty(B_R)}=0.
\end{equation}
By \eqref{eq:4.1} and \eqref{eq:4.7}, we can find sequences $\{t_j^{(1)}\}_{j=1}^\infty$ and $\{t_j^{(2)}\}_{j=1}^\infty$ such that
$t_j^{(i)}\nearrow T$ as $j\to\infty$ for $i=1,2$
and
\begin{equation*}
\lim_{j\to\infty} (T-t_j^{(1)})^\frac{2+\sigma}{2(p-1)} M(t_{j}^{(1)}) = \infty \quad \mbox{and} \quad 
\lim_{j\to\infty} (T-t_j^{(2)})^\frac{2+\sigma}{2(p-1)} M(t_j^{(2)}) = 0.
\end{equation*}
Due to Lemma~\ref{Lemma:4.2} and the definition of $M(t)$, we see that
\begin{equation*}
\lim_{j\to\infty} (T-t_j^{(1)})^\frac{2+\sigma}{2(p-1)} u(0,t_j^{(1)}) = \infty \quad \mbox{and} \quad 
\lim_{j\to\infty} (T-t_j^{(2)})^\frac{2+\sigma}{2(p-1)} u(0,t_j^{(2)}) = 0.
\end{equation*}
Let $\Phi_m$ be the backward self-similar solution as in \eqref{eq:BSSS} blowing up at $t=T$.
By the definitions of $\{t_j^{(i)} \}_{j=1}^{\infty} \,\, (i=1,2)$ and $\Phi_m$,
we can construct another sequence $\{t_j^{(3)}\}_{j=1}^\infty$ such that $t_j^{(3)} \nearrow T$ as $j\to \infty$ and
for sufficiently large $j$, 
\begin{equation*}
\begin{split}
& u(0,t_j^{(3)}) >  \Phi_m(0,t_j^{(3)} ) \quad \mbox{if} \quad  j \,\,  \mbox{is odd},\\
& u(0,t_j^{(3)}) <  \Phi_m(0,t_j^{(3)} ) \quad \mbox{if} \quad  j \,\,  \mbox{is even}.\\
\end{split}
\end{equation*}
This fact and the intermediate value theorem imply that 
there exists a sequence $\{t_j\}_{j=1}^\infty$
such that $t_j^{(3)} < t_j < t_{j+1}^{(3)}$ and
$u(0,t_j) = \Phi_m(0, t_j)$ for all sufficiently large $j$.
Since $\partial_r u(0,t_j) = \partial_r \Phi_m(0, t_j)=0$, 
$\mathcal{Z}_{[0,R]} (u(0,t) - \Phi_m(0, t))$
must drop at $t=t_j$ for every $j$.
However, this contradicts Proposition~\ref{Proposition:zeronumber} (i). 
We thus obtain \eqref{eq:4.6} and the proof is complete.
\end{proof}

The following lemma describes the behavior of the maximum point of $u$ when  $u$ exhibits Type II blow-up.

\begin{lemma}
\label{Lemma:4.4}
Assume $p>p_{\rm S}(\sigma)$ and let $u$ be a solution of problem~\eqref{eq:Henon} in $B_R\times(0,T)$, where $T\in (0,\infty)$.
If $u$ exhibits Type II blow-up, then one has 
\begin{equation}
\label{eq:4.8}
\limsup_{t\nearrow T} \frac{|x(t)|}{\sqrt{T-t}}<\infty,
\end{equation}
where $x(t)$ is as in \eqref{eq:defofxt}.
\end{lemma}

\begin{proof}
By Lemma~\ref{Lemma:4.3} we may assume that there exists $\delta>0$ such that $\delta (T-t)^{-(2+\sigma)/2(p-1)}  \le M(t)$
for all $t\in (0,T)$.
It follows from Proposition~\ref{Proposition:SDE} and Lemma~\ref{Lemma:4.3} that
\begin{equation}
\label{eq:4.9}
\begin{split}
 \delta(T-t)^{-\frac{2+\sigma}{2(p-1)}}|x(t)|^{\frac{2+\sigma}{p-1}}
&\le M(t)|x(t)|^{\frac{2+\sigma}{p-1}}
= |x(t)|^{\frac{2+\sigma}{p-1}} u(x(t),t)\\
&\le C\left(1+t^{-\frac{1}{p-1}}|x(t)|^{\frac{2}{p-1}} +\left(\frac{|x(t)|}{\sqrt{T-t}}\right)^\frac{2}{p-1}\right),
\end{split}
\end{equation}
which implies that
\begin{equation*}
\begin{split}
  \delta\left(\frac{|x(t)|}{\sqrt{T-t}}\right)^\frac{2+\sigma}{p-1}
&\le C\left(1+|x(t)|^{\frac{2}{p-1}} +\left(\frac{|x(t)|}{\sqrt{T-t}}\right)^\frac{2}{p-1}\right)   
\end{split}
\end{equation*}
for $T/2<t<T$.
Since $\sigma>0$, we see that
\[
\limsup_{t\nearrow T} \frac{|x(t)|}{\sqrt{T-t}}< \infty
\]
must hold.
The proof is now complete.
\end{proof}

We are now ready to prove Proposition~\ref{Proposition:rescale}.

\begin{proof}[Proof of Proposition~{\rm \ref{Proposition:rescale}}]
First, we assume that $u_t(0,t)<0$ near $t=T$.
This implies that $u(0,t)$ is uniformly bounded near $t=T$ and $u(0,t)/M(t)\to 0$ as $t\nearrow T$.
By Lemma~\ref{Lemma:4.2}, the blow-up of $u$ must be of Type I.
This contradicts \eqref{eq:4.1}.
Thus, we obtain $u_t(0,t)>0$ near $t=T$.

We claim that
\[
\liminf_{t\nearrow T} \frac{u_t(0,t)}{M(t)^{\frac{2p+\sigma}{2+\sigma}}} = 0.
\]
If the assertion is not true, then there exist $\delta>0$ and $\tau_*\in(0,T)$ such that
\[
\frac{u_t(0,t)}{u(0,t)^{\frac{2p+\sigma}{2+\sigma}}}\ge \frac{u_t(0,t)}{M(t)^{\frac{2p+\sigma}{2+\sigma}}} \ge \delta
\]
for all $t\in (\tau_*,T)$.
This implies that 
there exists $C>0$ such that
\begin{equation}
\label{eq:4.10}
u(0,t) \le C(T-t)^{-\frac{2+\sigma}{2(p-1)}} \quad \mbox{for} \quad t\in (\tau_*,T).
\end{equation}
By \eqref{eq:4.1}, we can find a sequence $\{t'_j\}_{j=1}^\infty\subset (0,T)$ such that $t'_j\nearrow T$ as $j\to \infty$ and
\[
\lim_{j\to \infty}(T-t'_j)^\frac{2+\sigma}{2(p-1)}M(t'_j) = \infty.
\]
This together with \eqref{eq:4.10} implies that 
\[
\frac{u(0,t'_j)}{M(t'_j)} \le \frac{C}{(T-t'_j)^\frac{2+\sigma}{2(p-1)} M(t'_j)} \to 0 \quad \mbox{as}\quad j\to \infty.
\]
This and Lemma~\ref{Lemma:4.2} contradict the definition of $\{t'_j\}_{j=1}^\infty$.
Thus, the claim follows and   there exists a sequence $\{t_j\}_{j=1}^\infty\subset (0,T)$ with $t_j\nearrow  T $  such that \eqref{eq:4.2} holds.

By \eqref{eq:4.8} and \eqref{eq:4.9}, we see that there exists $C>0$ such that
\begin{equation}
\label{eq:4.11}
M(t) \le C |x(t)|^{-\frac{2+\sigma}{p-1}}\quad \mbox{for all} \quad t\in(T/2,T).
\end{equation}
It follows from Lemmas~\ref{Lemma:CisI}, \ref{Lemma:HI21}, and the comparison principle that
\begin{equation}
\label{eq:4.12}
 \begin{split}
u(x,t) 
&= [e^{(t-t_j)\Delta}u(\cdot,t_j)](x)
+ 
\int_{t_j}^t [e^{(t-s)\Delta}|\cdot|^{\sigma}u(\cdot,s)](x) \, \di s\\
&\le M(t_j) + \int_{t_j}^t M(s)^p
\int_{\mathbb{R}^N} G(x-y,t-s) |y|^{\sigma} \, \di y \di s\\
 \end{split}   
\end{equation}
for all $x\in B_R$ and $t_j<t<T$. 
Assume that $|y|\ge 2|x(t)|$. Then
$|x(t)-y| \ge |y|-|x(t)| \ge |y|/2$.
Since $\tau\mapsto \tau^\sigma \exp(-\tau^2/32)$ is bounded in $[0,\infty)$,
we have
\begin{equation*}
\begin{split}
&\int_{|y|\ge 2|x(t)|} G(x(t)-y,t-s) |y|^{\sigma} \, \di y\\
&\le C(t-s)^{-\frac{N}{2}} \int_{|y|\ge 2|x(t)|} \exp\left(-\frac{|x(t)-y|^2}{4(t-s)}\right)|y|^{\sigma} \, \di y\\
&\le C(t-s)^\frac{\sigma}{2} \int_{|y|\ge 2|x(t)|} (t-s)^{-\frac{N}{2}}\exp\left(-\frac{|y|^2}{32(t-s)}\right)\\
&\qquad\qquad\qquad\qquad\cdot
\left(\frac{|y|}{\sqrt{t-s}}\right)^{\sigma}
\exp\left(-\frac{|y|^2}{32(t-s)}\right) \, \di y\\
&\le C  (t-s)^\frac{\sigma}{2} \int_{\mathbb{R}^N}
(t-s)^{-\frac{N}{2}}\exp\left(-\frac{|y|^2}{32(t-s)}\right) \, \di y
\le C  (t-s)^\frac{\sigma}{2}. 
\end{split}
\end{equation*}
On the other hand, by \eqref{eq:4.11} we have
\begin{equation*}
\begin{split}
\int_{|y|< 2|x(t)|} G(x(t)-y,t-s) |y|^{\sigma} \, \di y 
&\le C|x(t)|^\sigma \int_{|y|< 2|x(t)|} G(x(t)-y,t-s)\, \di y\\
& 
\le CM(t)^{-\frac{p-1}{2+\sigma}\sigma}. 
\end{split}
\end{equation*}
These together with \eqref{eq:4.12} imply that
\begin{equation*}
\begin{split}
M(t) \le M(t_j) + C\int_{t_j}^t (t-s)^\frac{\sigma}{2}M (s)^p \, \di s
+ C M(t)^{-\frac{p-1}{2+\sigma}\sigma}
\int_{t_j}^t M (s)^p \, \di s 
\end{split}
\end{equation*}
for all $t_j<t<T$.
Choosing
\[
t=\tilde{t}_j := \min\{t' \in (t_j,T); M(t') = 2M(t_j) \} \in (t_j, T),
\]
we have
\begin{equation*}
\begin{split}
2M(t_j)
&\le  M(t_j) + CM(t_j)^p\int_{t_j}^{\tilde{t}_j}(\tilde{t}_j-s)^\frac{\sigma}{2} \, \di s + 
C(\tilde{t}_j-t_j) M(t_j)^{p-\frac{p-1}{2+\sigma}\sigma} \\
&\le  M(t_j) + C (\tilde{t}_j - t_j )^{1+\frac{\sigma}{2}}M(t_j)^p + 
C (\tilde{t}_j-t_j)M(t_j)^{p-\frac{p-1}{2+\sigma}\sigma}.
\end{split}
\end{equation*}
Then
\[
1\le C\left((\tilde{t}_j-t_j)^{\frac{2+\sigma}{2(p-1)}} M(t_j)\right)^{p-1}
+
C\left((\tilde{t}_j-t_j)^{\frac{2+\sigma}{2(p-1)}} M(t_j)\right)^{\frac{2(p-1)}{2+\sigma}},
\]
which implies that there exists a constant $c>0$ independent of $j$ such that
\[
c\le (\tilde{t}_j-t_j)^{\frac{2+\sigma}{2(p-1)}} M(t_j).
\]
Therefore, we arrive at
\begin{equation}
\label{eq:4.13}
t_j + s^*M(t_j)^{-\frac{2(p-1)}{2+\sigma}} \le \tilde{t}_j \quad \mbox{and} \quad M(t) \le 2M(t_j), \quad t_j \le t\le \tilde{t}_j.
\end{equation}
where $s^*>0$ is depending only on $N$, $p$, and $\sigma$, in particular, independent of $j$.

Let $v_j$ be as in \eqref{eq:4.3}.
By the same argument as in Lemma~\ref{Lemma:4.2}, we see that
\begin{equation*}
\partial_sv_j(0, 0) = \frac{\lambda_j^2}{M(t_j)} u_t(0,t_j) = \frac{u_t(0,t_j)}{M(t_j)^\frac{2p+\sigma}{2+\sigma}} \to 0 \quad \mbox{as} \quad t\nearrow T.
\end{equation*}
Let $D=(0,\infty)\times(-\infty,s^*)$, where $s^*>0$ is as in \eqref{eq:4.13}.
If $\{t_j\}_{j=1}^\infty$ satisfies
\[
M(t_j) = \max_{t\in[0,t_j]} M(t) \quad \mbox{for all} \quad j,
\]
then by \eqref{eq:4.13} we see that $0\le {v}_j \le 2$
in $(0, \lambda_j^{-1}R)\times (-\lambda_j^{-2}t_j,s^*)$.
If not, there exists $t_j' \in [0,t_j)$ such that
\[
M(t_j') = \max_{t\in [0,t_j)} M(t) > M(t_j).
\]
If follows from Lemma~\ref{Lemma:4.2} that
there exists $K \ge 1$ such that
\[
M(t_j') \le K u(0,t_j') < K u(0,t_j) \le K M(t_j)
\quad\mbox{for sufficiently large} \quad j,
\]
where we used $u_t(0,t)>0$ near $t=T$.
This implies that $0\le v_j (r,s) \le K$  in $(0, \lambda_j^{-1}R)\times (-\lambda_j^{-2}t_j,0)$.
Combining this and \eqref{eq:4.13}, we see that
$0\le v_j (r,s) \le \max\{2,K\}$  in $(0, \lambda_j^{-1}R)\times (-\lambda_j^{-2}t_j,s^*)$.
Summarizing the above we see that $0\le {v}_j \le \max\{2,K\}$
in $(0, \lambda_j^{-1}R)\times (-\lambda_j^{-2}t_j,s^*)$ and
 ${v}_j$ satisfies 
\begin{equation*}
	\displaystyle{\partial_s {v}_j = \partial_{rr} v_j + \frac{N-1}{r} \partial_r v_j + r^\sigma v_j^p} \,\quad 
\mbox{for} \,\,\,\,
r\in \left(0, \frac{R}{\lambda_j}\right),\,\, s\in \left(-\frac{t_j}{\lambda_j^2}, \frac{T-t_j}{\lambda_j^2}\right)
\end{equation*}
and
\[
\partial_r \tilde{v}_j\left(0,s \right)=0 \quad \mbox{for}\,\,\,\,
s\in \left(-\frac{t_j}{\lambda_j^2}, \frac{T-t_j}{\lambda_j^2}\right).
\]
Therefore, by interior parabolic estimates, it follows that after extracting a subsequence, still denoted by $\{t_j\}$,
$v_j$ converges in $C_{\rm loc}^{2+\alpha, 1+\alpha/2}(D)$ to a nonnegative solution $v_\infty$ of
\begin{equation*}
\left\{
\begin{array}{ll}
	\displaystyle{\partial_s v_\infty = \partial_{rr} v_\infty + \frac{N-1}{r} \partial_r v_\infty +  r^\sigma v_\infty^p},\quad &  (r,s)\in D, \vspace{3pt}\\
   \partial_r v_\infty(0,s)  =0, \quad & s\in (-\infty,s^*), \vspace{3pt}\\
    v_\infty(0,0)=1, \,\, \partial_sv_\infty(0,0) =0.\\
\end{array}
\right.
\end{equation*}
Note that $0\le v_\infty \le \max\{2,K\}$ in $\tilde{D}$.

We claim that $\partial_s v_\infty(\cdot,0) \equiv 0$.
Suppose not.
Then there exist $A>0$ and $\varepsilon\in(0,s^*)$ such that
\begin{equation}
\label{eq:4.14}
\partial_s v_\infty(A,s) \neq 0, \quad |s| \le \varepsilon.
\end{equation}
By the above arguments, 
we see that  $\partial_s v_\infty(\cdot,0)$ has a degenerate zero at $r=-P$. It then follows from Proposition~\ref{Proposition:zeronumber} that the zero number of $\partial_s v$ on $[0,A]$ drops at $s=0$.
Namely, we can fix $-\varepsilon<s_1<0<s_2<\varepsilon$
such that $\partial_s v_\infty(\cdot,s_i)$ has only simple zeros on $[0,A]$ and such that
\[
\mathcal{Z}_{[0,A]}(\partial_s v_\infty(\cdot,s_1)) \ge \mathcal{Z}_{[0,A]}(\partial_s v_\infty(\cdot,s_2))+1.
\]
We deduce that for $j$ large enough,
\[
\mathcal{Z}_{[0,A]}(\partial_s v_j(\cdot,s_1)) \ge \mathcal{Z}_{[0,A]}(\partial_s v_j(\cdot,s_2))+1.
\]
Hence 
\begin{equation}
\label{eq:4.15}
\mathcal{Z}_{[0,A\lambda_j]}(u_t(\cdot,t_j+\lambda_j^2 s_1)) \ge \mathcal{Z}_{[0,A\lambda_j]}(u_t(\cdot,t_j + \lambda_j^2s_2))+1.
\end{equation}
Since on the other hand, \eqref{eq:4.14} implies that $u_t (A\lambda_j, t_j+\lambda_j^2 s) \neq 0$ for $|s|\le \varepsilon$ and
\begin{equation}
\label{eq:4.16}
\mathcal{Z}_{[A\lambda_j,R]}(u_t(\cdot,t_j+\lambda_j^2 s_1)) \ge \mathcal{Z}_{[A\lambda_j,R]}(u_t(\cdot,t_j+\lambda_j^2 s_2)).
\end{equation}
By \eqref{eq:4.15} and \eqref{eq:4.16}, we deduce that
$\mathcal{Z}_{[0,R]}(u_t(\cdot,t_j+\lambda_j^2 s_1)) \ge \mathcal{Z}_{[0,R]}(u_t(\cdot,t_j+\lambda_j^2 s_2)) +1$, which contradicts Lemma~\ref{Lemma:4.1}.
It follows $\partial_s v_\infty(\cdot,0) \equiv 0$.
Hence, $v_\infty(\cdot,0)$ satisfies
\begin{equation*}
\left\{
\begin{array}{ll}
	\displaystyle{\partial_{rr} v_\infty + \frac{N-1}{r} \partial_r v_\infty +  r^\sigma v_\infty^p} = 0, \quad \,\, r\in \left(0, \infty\right)\,\, s=0,\vspace{3pt}\\
    \partial_r v_\infty(0,0) = \partial_r v_\infty(0,0)=0, \,\, v_\infty(0,0)=1.\\
\end{array}
\right.
\end{equation*}
We can obtain the desired conclusion.
Thus, the proof is complete.
\end{proof}

At the end of this section, we shall provide a proof of Theorem~\ref{Theorem:D}.

\begin{proof}[Proof of Theorem~{\rm \ref{Theorem:D}}]
Let  $p_{\rm S}(\sigma) < p< p_{\rm JL}(\sigma)$ and let $u$ be a nonnegative radially symmetric solution of problem~\eqref{eq:Henon} which blows up at $t=T\in(0,\infty)$.
Due to Proposition~\ref{Proposition:lowerU}, we have already obtained the lower estimate of \eqref{eq:Theorem:D2}. 
The upper estimate of \eqref{eq:Theorem:D2} follows at once from the Type I estimate \eqref{eq:Theorem:D1} 
and Proposition~\ref{Proposition:Phan13}(with \eqref{eq:Phan13} applied for $|x| \geq \sqrt{T-t}$). 
The proof of Theorem~{\rm \ref{Theorem:D}} is concluded as soon as the blow-up is shown to be of Type I.
Suppose that the blow-up is of Type II.
By Proposition~\ref{Proposition:SS}, we have
\begin{equation}
\mathcal{Z}_{[0,\infty]} (U_1 - U_\infty) = \infty.
\end{equation}
 We see from Proposition~\ref{Proposition:rescale} that
 \begin{equation}
 \lim_{j\to \infty}Z_{[0,R/\lambda_j]} \left(\frac{1}{M(t_j)} u( \lambda_j r, t_j) - U_\infty \left( r\right)\right) = \infty.
\end{equation}
Then: 
 \begin{equation*}
 \begin{split}
 & \lim_{j\to \infty}\mathcal{Z}_{[0,R/\lambda_j]} \left( u(\lambda_j r, t_j) - U_\infty \left(\lambda_j r\right)\right) = \infty\\
 &\Rightarrow \lim_{j\to \infty}\mathcal{Z}_{[0,R]} \left( u(r, t_j) - U_\infty \left(r\right)\right) = \infty.
 \end{split}
\end{equation*}
This contradicts Proposition~\ref{Proposition:zeronumber}.
Therefore, the blow-up must be of Type I and the proof is complete.
\end{proof}

\begin{remark}
In the case of $p=p_{\rm S}(\sigma)$,
the arguments as in the proof of Proposition~{\rm \ref{Proposition:rescale}} cannot be applied since the backward self-similar solution as in \eqref{eq:BSSS} does not exist (see \cite[Section~5]{FT00} and Lemma~{\rm \ref{Lemma:FT00}} below).
However, if there exists a solution $\Phi$ of problem~\eqref{eq:Henon} in $\mathbb{R}^N$ satisfying
\[
\limsup_{t\nearrow T} (T-t)^{\frac{2+\sigma}{2(p-1)}} \Phi(0,t) \in (0, \infty),
\]
by replacing $\Phi_m$ in the proof of Lemma~{\rm \ref{Lemma:4.3}} with $\Phi$, we can prove Proposition~{\rm \ref{Proposition:rescale}} even in the case of $p=p_{\rm S}(\sigma)$.
\end{remark}

\section{Classification of threshold solutions}
\subsection{Possibilities of behaviors of threshold solutions}
The threshold solutions constructed in Section~3   
under the assumption $N\geq 3$ and $p_{\mathrm{S}}(\sigma ) < p < p_{\mathrm{JL}}(\sigma )$
have played a central role in Theorems \ref{Theorem:A} and \ref{Theorem:C}. In fact,
threshold solutions 
can be constructed for every $N\geq 1$ and $1+\sigma/N<p<p_{\mathrm{JL}}(\sigma )$. In this section, we classify their asymptotic behaviors for $N\ge1$ and
$1+\sigma/N<p\le p_{\mathrm{S}}(\sigma )$ for completeness.

Three possibilities arise for behaviors of $u^*$.
\begin{itemize}
   \item[(GB)]  $u^*$ exists globally-in-time and is uniformly bounded;
    \item[(G)]  $u^*$ grows up, {\it i.e.}, $\|u^*(\cdot, t)\|_{L^\infty(B_R)}< \infty$ for all $t>0$ and
    \[
    \limsup_{t\to \infty}\|u^*(\cdot, t)\|_{L^\infty(B_R)} = \infty;
    \]
    \item[(B)]  $u^*$ blows up in  finite time.
\end{itemize}
In what follows, we omit the star superscript on $u^*$ for simplicity.

\begin{theorem}
\label{Theorem:TSblowsup}
Assume $N\ge 3$, $1+\sigma/N < p< p_{\rm JL}(\sigma)$, and $\sigma>0$.
Let $g \in C(\overline{B_R}) \cap H^1_0(B_R)$ be a nonnegative radially symmetric function with $g\not\equiv 0$.
Then  threshold solution $u$ satisfies the following:
\begin{itemize}
    \item[(i)] If $1+\sigma/N<p<p_{\rm S}(\sigma)$, then only the case~{\rm (GB)} occurs;
    \item[(ii)] If $p=p_{\rm S}(\sigma)$, then only the case~{\rm (G)} occurs and the blow-up point is the origin;
    \item[(iii)] If $p_{\rm S}(\sigma)<p<p_{\rm JL}(\sigma)$, then only the case~{\rm (B)} occurs,
    the blow-up is of Type I, and the blow-up point is the origin.
\end{itemize}
\end{theorem}

We have already proved case~(iii) in Theorems~\ref{Theorem:A} and \ref{Theorem:D}.
Therefore, it suffices to consider cases~(i) and (ii).
We prepare several lemmas.
The following lemma, which is proved in \cite[Section~5]{FT00}, is a key to the proof of Theorem~\ref{Theorem:TSblowsup}.
\begin{lemma}
\label{Lemma:FT00}
Assume $N\ge 1$, $1<p\le p_{\rm S}(\sigma)$, and $\sigma>0$.
Then there exist no positive bounded solutions of 
\begin{equation}
\label{eq:BSS}
\Delta w - \frac{1}{2} y \cdot \nabla w -\frac{2+\sigma}{2(p-1)} w + |y|^\sigma w^p=0 \quad \mbox{in} \quad \mathbb{R}^N.
\end{equation}
In other words, the only bounded solution is $w \equiv 0$.
\end{lemma}

\begin{lemma}
\label{Lemma:TS1}
Assume $p>1+\sigma/N$ and let $u$ be a threshold solution of problem~\eqref{eq:Henon}.
If $u$ blows up at $t=T$,
then
\[
\limsup_{t\nearrow T} |x(t)|^\frac{2+\sigma}{p-1} M(t) < \infty
\quad \mbox{and} \quad \limsup_{t\nearrow T} \frac{|x(t)|}{\sqrt{T-t}} < \infty,
\]
where $x(t) \in B_R$ is as in \eqref{eq:defofxt}.
\end{lemma}
\begin{proof}
By Proposition~\ref{Proposition:SDE}, we have
\[
|x(t)|^\frac{2+\sigma}{p-1} M(t) = |x(t)|^\frac{2+\sigma}{p-1}u(x(t),t)\le C
\]
for all $t\in(T/2,T)$. Then the first claim is proved.

By Proposition~\ref{Proposition:BRforThreshold} and Lemma~\ref{Lemma:TS1}, there exists $\delta>0$ such that
\[
\delta (T-t)^{-\frac{2+\sigma}{2(p-1)}} \le u(x(t),t) \le C|x(t)|^{-\frac{2+\sigma}{p-1}}
\]
for all $t\in(T/2,T)$. The second claim is then proved.
The proof is complete.
\end{proof}

\subsection{Subcritical case}

When $1+\sigma/N < p< p_{\rm S}(\sigma)$, we know from \cite[Theorem~1.2~(i)]{Phan17} that the blow-up must be of Type I.
Note that although  \cite[Theorem~1.2~(i)]{Phan17} is for the estimate of solutions in the whole space $\mathbb{R}^N$, the same conclusions can be reached for solutions in bounded domains by  similar arguments.

\begin{proof}[Proof of Theorem~{\rm \ref{Theorem:TSblowsup}} in the subcritical case.]
We shall show that case~(B) never occurs.
Suppose that  threshold solution $u$ blows up at $t=T$.
Note that by Proposition~\ref{Proposition:SDE} the blow-up point is only the origin.
First, we claim that 
\begin{equation}
\label{eq:5.0}
\limsup_{t\nearrow T}(T-t)^{\frac{2+\sigma}{2(p-1)}}\|u(\cdot,t)\|_{L^\infty(B_R)} =0.
\end{equation}
By Lemma~\ref{Lemma:TS1}, 
there exists $K>0$ such that 
\begin{equation}
\label{eq:5.1}
|x(t)| \le K \sqrt{T-t} \quad \mbox{near} \quad t=T.
\end{equation}
This implies that one of the blow-up points of $u$
is the origin.
Set 
\[
w(y,s) := (T-t)^\frac{2+\sigma}{2(p-1)} u(y\sqrt{T-t}, s) ,\qquad s= -\log (T-t).
\]
Then $w$ is a radially symmetric solution of 
\begin{multline*}
w_s = \Delta w - \frac{1}{2} y \cdot \nabla w -\frac{2+\sigma}{2(p-1)} w + |y|^\sigma w^p, \quad (y,s) \in B_{R/\sqrt{T-t}} \times (-\log T, \infty) \\
\mbox{with } \,\, w(y,s)=0 \, \mbox{ on } \partial B_{R/\sqrt{T-t}} \times (-\log T, \infty).
\end{multline*}
It follows from \cite[Theorem~1.2~(i)]{Phan17} that
$w$ is uniformly bounded.
We introduce the energy functional 
\[
\mathcal{E}[w] := \int_{B_{R/\sqrt{T-t}}} \left(\frac{1}{2}|\nabla w|^2 + \frac{2+\sigma}{p-1} |w|^2 - \frac{1}{p+1} |y|^\sigma |w|^{p+1}\right) \rho \, \di y,
\]
where $\rho(y) = \exp(-|y|^2/4)$.
Simple calculations yield 
\begin{equation*}
\frac{\di}{\di s} \mathcal{E}[w] \le - \int_{B_{R/\sqrt{T-t}}} w_s^2 \rho \, \di y \le 0. 
\end{equation*}
By the standard dynamical system argument with Lyapunov functional, the $\omega$-limit set with respect to $C^2_{\rm loc}$ topology is included in the set of nonnegative bounded solutions to \eqref{eq:BSS}.
However, since $1+\sigma/N < p< p_{\rm S}(\sigma)$,
it follows from Lemma~\ref{Lemma:FT00} that the solution to this equation is only $0$.
Thus, $w(y,s) \to 0$ uniformly in compact sets in $\mathbb{R}^N$.
This together with \eqref{eq:5.1} implies that
\begin{equation*}
\begin{split}
\limsup_{t\nearrow T}(T-t)^\frac{2+\sigma}{2(p-1)}\|u(\cdot,t)\|_{L^\infty(B_R)} 
&= \limsup_{t\nearrow T} (T-t)^\frac{2+\sigma}{2(p-1)} u(x(t),t)\\
&\le  \limsup_{s\to \infty} \|w(\cdot, s)\|_{L^\infty (B_K)}
= 0
\end{split}
\end{equation*}
Since $u$ is a threshold solution,
it follows from Propositions~\ref{Proposition:SDE} and \ref{Proposition:lowerU} that 
\begin{equation*}
\begin{split}
C(T-t)^{-\frac{1}{p-1}} \le |X(t)|^\frac{\sigma}{p-1} u(X(t),t) \le C(1 + t^\frac{1}{p-1} + |X(t)|^{-\frac{2}{p-1}})
\end{split}
\end{equation*}
for $T/2<t<T$, 
where $X(t)\in B_R$ is the maximum point of $|\cdot|^{\sigma/(p-1)}u(\cdot, t)$: 
\[
\||\cdot|^\frac{\sigma}{p-1}u(\cdot, t)\|_{L^\infty(B_R)}=|X(t)|^\frac{\sigma}{p-1}u(X(t), t) \quad \mbox{for} \quad \frac{T}{2}<t<T.
\]
Then:  
\begin{equation*}
|X(t)| \le C \sqrt{T-t} \quad \mbox{for} \quad \frac{T}{2} < t < T.
\end{equation*}
This together with \eqref{eq:5.0} implies that 
\begin{equation*}
\begin{split}
(T-t)^\frac{1}{p-1}\||\cdot|^\frac{\sigma}{p-1} u(\cdot,t)\|_{L^\infty(B_R)} 
\to 0
\end{split}
\end{equation*}
as $t\nearrow T$.
It then follows from Theorem~\ref{Theorem:B} that the origin is not a blow-up point.
This contradicts Proposition~\ref{Proposition:SDE}.
Thus, case~(B) never occurs.

We shall show that case~(G) never occurs but
this follows immediately from \cite[Theorem~1.5]{Phan13}.
Therefore, the only possible case is (GB), whence the result.
\end{proof}

\subsection{Critical case}

\begin{proof}[Proof of Theorem~{\rm \ref{Theorem:TSblowsup}} in the critical case]
First, assume that (B) occurs.
We claim that the blow-up must be of Type I.
If $\sigma=0$, this has been proved by Matano and Merle \cite[Theorem~1.7]{MM04}. We modify their arguments. 
Suppose that the blow-up is of Type II.
Since the solution to \eqref{eq:BSS} is only $0$ when $p=p_{\rm S}(\sigma)$, Proposition~\ref{Proposition:rescale} cannot be proved by the same arguments as in the case of $p_{\rm S}(\sigma)<p<p_{\rm JL}(\sigma)$ (see the proof of Lemma~\ref{Lemma:4.3}).
However, thanks to Lemma~\ref{Lemma:TS1}, we see that Lemmas~\ref{Lemma:4.2}, \ref{Lemma:4.3}, and \ref{Lemma:4.4} are valid for threshold solutions.
This implies that Proposition~\ref{Proposition:rescale} is valid 
for threshold solution $u$, even in the case of $p=p_{\rm S}(\sigma)$.
Set \[
L:= \liminf_{t\nearrow T} \frac{u(0,t)}{M(t)}\in (0,1].
\]
Note that $L>0$ implies that
$u(0,t) \ge L M(t)$ near $t=T$.
Due to the strong maximal and the comparison principles, there exist $\delta_0>0$ and $t_0\in[0,T)$ such that
\[
u\left(\frac{R}{2},t\right) \ge \delta_0 \quad\mbox{for all} \quad t\in[t_0,T).
\]
Taking $m>0$ sufficiently large, we see from Proposition~\ref{Proposition:SS} that 
\[
U_m\left(\frac{R}{2}\right) \le \frac{\delta_0}{2} \quad\mbox{and} \quad \mathcal{Z}_{[0,R/2]}(u(\cdot,t_0)- U_m) = 1.
\]
Now let $\{t_j\}$ be as in Proposition~\ref{Proposition:rescale}
and take $m= LM(t_j)/2$.
Since $M(t_j) \to \infty$ as $j\to\infty$, taking sufficiently large $j$ if necessary, the above inequalities hold with this $m$. 
Then since $u(0,t_j)\ge LM(t_j) \to \infty$ as $j\to\infty$, we see that 
\[
\mathcal{Z}_{[0,R/2]}(u(\cdot,t_j)-U_m) = 0,
\]
and
\begin{equation}
u(r,t_j) > U_m(r) \quad \mbox{for all} \quad r\in\left[0,\frac{R}{2}\right].
\end{equation}
Therefore, 
\begin{equation}
\begin{split}
\frac{1}{M(t_j)} u(\lambda_jr, t_j) 
&> 
\frac{1}{M(t_j)} U_m( \lambda_jr)=U_{L/2}\left(r\right)
\quad\mbox{for all}\quad r\in \left[0, \frac{R}{2\lambda_j}\right].
\end{split}
\end{equation}
It follows from Proposition~\ref{Proposition:rescale} that
\[
U_1(r) \ge U_{L/2}(r) \quad \mbox{for all} \quad r \in [0,\infty).
\]
This contradicts Proposition~\ref{Proposition:SS}.
Therefore, the blow-up must be of Type I, and 
by the same arguments as in the subcritical case,
case~(B) never occurs.

Assume (GB). By the standard dynamical system argument with Lyapunov functional, the $\omega$-limit set with respect to $C^2_{\rm loc}$ topology is included in the set of nonnegative steady states.
When $p=p_{\rm S}(\sigma)$, however, there exist no positive steady states (see {\it e.g.}, \cite{Ni86}).
This implies that $u\to 0$ as $t\to\infty$,
that is, $u(\cdot,0) = \mu^*g \in A$.
Since $A$ is an open set in $X$, we can find $\mu >\mu^*$ such that $\mu g\in A$.
This is a contradiction. Therefore, case (GB) never occurs
and the only possible case is (G), whence the result.
\end{proof}

\appendix
\section{Minimal $L^1$-continuation}
We have constructed threshold solutions $u^*$ of problem~\eqref{eq:Henon} in Section~3. 
In this appendix, we investigate the continuation beyond the blow-up time, which played a central role in \cite{GS11}, 
although this fact is not necessary to prove our main results.

\subsection{A priori estimates}
In this subsection we establish a priori estimates of solutions of problem~\eqref{eq:Henon}, 
using  property \eqref{eq:I1}. 

\begin{lemma}
\label{Lemma:AE}
Assume $N\ge 3$, $p>p_{\rm S}(\sigma)$, and $\sigma>0$.
Let $u$ be a nonnegative  global-in-time solution of problem~\eqref{eq:Henon}.
Then the following  a priori estimates hold:
\begin{equation}
\label{eq:AE1}
\int_{B_R} u(x,t)^2 \, \di x \le \left[\frac{2}{c_*} J[u_0]\right]^\frac{2}{p+1},
\end{equation}
and
\begin{equation}
\label{eq:AE2}
\int_{0}^t \int_{B_R} |x|^\sigma u(x,t)^{p+1} \,\di x \di t
\le \frac{p+1}{p-1} \left(\frac{1}{2} \left[\frac{2}{c_*} J[u_0]\right]^\frac{2}{p+1} + 2 J[u_0] t\right),
\end{equation}
for all $t\ge0$, where $J[u]$ denotes the energy functional defined in \eqref{energy:J} and 
\[
c_* := \frac{p-1}{p+1} \left(\int_{B_R} |x|^{-\frac{2\sigma}{p-1}} \, \di x\right)^{-\frac{p-1}{2}}.
\]
\end{lemma}
\begin{proof}
First, we shall prove \eqref{eq:AE1}.
In terms of the monotonicity \eqref{eq:I1} of $J[u(\cdot, t)]$, it suffices to show
\begin{equation}
\label{eq:LemAE1}
\left(\int_{B_R} u (x,t)^2 \, \di x\right)^{\frac{p+1}{2}} \le \frac{2}{c_*}J[u(\cdot,t)] \quad \mbox{for all} \quad t\ge0. 
\end{equation}
Assume that there exists $\tau\ge0$ such that \eqref{eq:LemAE1} does not hold.
Set
\[
Y(t) := \int_{B_R} u(x,t)^2 \, \di x. 
\]
Noting that $p_S(\sigma) > 1+ 2\sigma/N$, by H\"{o}lder's inequality we have
\begin{equation*}
\frac{p-1}{p+1} \int_{B_R} |x|^\sigma u(x,t)^{p+1} \, \di x
\ge
c_* Y(t)^\frac{p+1}{2}.
\end{equation*}
This together with \eqref{eq:I1}  implies that
\begin{equation*}
\begin{split}
\frac{1}{2} Y'(t) 
&= -2 J[u(\cdot,t)] + c_* Y(t)^\frac{p+1}{2}
\ge -2 J[u(\cdot,\tau)] + c_* Y(t)^\frac{p+1}{2}
\end{split}
\end{equation*}
for $t\ge \tau$.
Setting
\[
h(Y) := -2 J[u(\cdot,\tau)] + c_* Y^\frac{p+1}{2} \quad \mbox{for} \quad Y>0,
\]
we see that $h$ is increasing for $Y>0$ and
$h(Y(\tau)) =  -2 J[u(\cdot,\tau)] + c_* Y(\tau)^{(p+1)/2}>0$.
Let $Y^*>0$ be such that $h(Y^*)=0$. Then we see that $Y(\tau) > Y^*$
and
\[
\int_{Y(\tau)}^\infty \frac{\di Y}{h(Y)} = 
\int_{Y(\tau)}^\infty \frac{\di Y}{-2 J[u(\cdot,\tau)] + c_* Y^\frac{p+1}{2}} < \infty.
\]
This implies that $Y(t)$ blows up in  finite time, a contradiction.
Thus, \eqref{eq:LemAE1} follows.

We shall prove \eqref{eq:AE2}.
Integrating the second identity in \eqref{eq:I1} over $(0,t)$, we get
\begin{equation*}
\begin{split}
&\frac{1}{2} \left[\int_{B_R} u(x,t)^2 \, \di x -\int_{B_R} u_0(x)^2\, \di x\right]
+ 2 \int_0^t J[u(\cdot,t)] \, \di t\\
&\qquad\qquad=
\frac{p-1}{p+1} \int_0^t\int_{B_R} |x|^\sigma u(x,t)^{p+1} \, \di x\di t.
\end{split}
\end{equation*}
Due to \eqref{eq:I1} and  \eqref{eq:AE1}, we have
\begin{equation*}
\begin{split}
\int_0^t \int_{B_R} |x|^\sigma u(x,t)^{p+1} \, \di x\di t
&\le \frac{p+1}{p-1} \left[ \frac{1}{2} \left[\frac{2}{c_*} J[u_0]\right]^\frac{2}{p+1}+ 2J[u_0] t\right]\\
\end{split}
\end{equation*}
for all $t>0$. This is the desired estimate \eqref{eq:AE2}.
Thus, the proof is complete.
\end{proof}

\subsection{Minimal $L^1$-continuation}
We set $f(x,u):=|x|^\sigma u^p$.
Note that $f_u(x,u)\ge0$ and $f(x,0)=0$.
Following \cite{MM09} and \cite{GS11}, we introduce the concept of 
minimal $L^1$-continuation of solutions to problem~\eqref{eq:Henon}.
\begin{definition}
The function $\tilde{u}$ is the minimal $L^1$-solution of problem~\eqref{eq:Henon} with initial data $u_0$ in the maximal existence time interval $[0,T)$, if there exists a sequence $\{\tilde{u}_{0,n}\}_{n=1}^\infty\subset C(\overline{B_R})$ with
\[
0\le \tilde{u}_{0,1} \le \tilde{u}_{0,2} \le \cdots \le  \tilde{u}_{0,n} \le \cdots \to u_0 \quad \mbox{in} \quad C(\overline{B_R})
\]
and $\tilde{u}_{0,n}\not\equiv u_0$ for all $n=1,2,\cdots$ such that the classical solution $\tilde{u}_n$ of problem~\eqref{eq:Henon} with initial data $\tilde{u}_{0,n}$
exists for all $t\in [0,T)$ and satisfies
\begin{equation}
\label{eq:L11}
\lim_{n\to \infty} \|\tilde{u}_n(\cdot,t) - \tilde{u}(\cdot,t)\|_{L^1(B_R)} = 0 \quad \mbox{for all} \quad t\in [0,T),
\end{equation}
and
\begin{equation}
\label{eq:L12}
\lim_{n\to \infty} \|f(\cdot, \tilde{u}_n(\cdot,t)) - f(\cdot, \tilde{u}(\cdot,t))\|_{L^1(B_R\times(0,t))} = 0 \quad \mbox{for all} \quad t\in [0,T).
\end{equation}
\end{definition}

Let $u(x, t; u_0)$ be a classical solution of the problem~\eqref{eq:Henon} which blows up in finite time $T>0$ and let $\tilde
{u}$ be the minimal solution $L^1$ with initial data $u_0$ in $[0,T^c)$  for some $T^c\ge T$.
The well-posedness of problem~\eqref{eq:Henon} implies that $\tilde{u}_n(\cdot,t)\to u(\cdot,t)$ 
for all $t\in [0,T)$.
We call $\tilde{u}$ as the minimal $L^1$-continuation of $u$. We say that the blow-up is complete if $T=T^c$ and is incomplete if $T<T^c$.
If $T^c=\infty$, we call $\tilde{u}$ as an $L^1$-global-in-time minimal continuation.

\begin{proposition}
\label{Proposition:A3}
Assume $N\ge3$, $p>p_{\rm S}(\sigma)$, and $\sigma>0$.
Let $g\in X\setminus\{0\}$.
Then  threshold solution $u^*(x,t)= u(x,t; \mu^* g)$ is bounded on $(0,\infty)$ in $L^1({B_R})$-norm and  unbounded on $(0,\infty)$ in $L^\infty$-norm.
\end{proposition}

\begin{proof}
The proof follows the arguments in \cite[Proposition~2]{GS11}.
Let $\mu \in (0,\mu^*)$ and denote $u_\mu(x,t) = u(x,t;\mu g)$.
We see from Proposition~\ref{Proposition:TSblowsup} that $u^*$ is unbounded on $(0,\infty)$ in $L^\infty$-norm in the case of $p_{\rm S}(\sigma)<p<p_{\rm JL}(\sigma)$.
In the case of $p\ge p_{\rm JL}(\sigma)$,
the same argument as in the proof of Proposition~\ref{Proposition:TSblowsup} shows that (GB) does not occur.
Hence, $u^*$ is unbounded on $(0,\infty)$ in $L^\infty$-norm.

 The comparison principle and the continuous dependence on initial values imply that $u_\mu$ is monotone increasing in $\mu$.
Hence, we see that
\[
u^*(x,t) = \lim_{\mu \nearrow \mu^*} u_\mu (x,t), \quad x\in B_R,\,\, t\in[0,T^*), 
\]
where $T^*\in (0,\infty]$ is the maximal existence time of $u^*$.
We shall show that $u^*$ is a minimal $L^1$-global solution.
By Fatou's lemma, H\"{o}lder's inequality, and \eqref{eq:AE1},
\begin{equation*}
\begin{split}
\int_{B_R} u^*(x,t) \, \di x 
&\le C \liminf_{\mu \nearrow \mu^*} \left(\int_{B_R} u_\mu(x,t)^2 \, \di x\right)^\frac{1}{2}
\le C\liminf_{\mu \nearrow \mu^*} J[\mu g]^\frac{1}{p+1}\\
\end{split}
\end{equation*}
for all $t>0$, where $C>0$ is independent of $t$. 
Since $J[\mu g]$ is bounded from above by 
$C(p, \mu^* )\| \nabla g \|_{L^2 (B_R )}$ for $\mu<\mu^*$, we obtain 
$\int_{B_R} u^*(x,t) \, \di x \le C (\mu^*)^{2/(p+1)}$.
Hence, $u^*= u_{\mu^*}$ exists globally in time as an $L^1$-solution.
The proof is complete.
\end{proof}

Finally, we shall show that  threshold solution $u^*$ is the minimal $L^1$-solution.
\begin{proposition}
Assume the same conditions as in Proposition~{\rm \ref{Proposition:A3}}. 
Then the threshold solution $u^*(x,t)= u(x,t; \mu^* g)$ is a 
minimal $L^1$-global-in-time solution.
\end{proposition}
\begin{proof}
Using the monotonicity of the sequence $u_\mu$ in $\mu$,
we deduce from the monotone convergence theorem, H\"{o}lder inequality, and \eqref{eq:AE1} that
\begin{equation*}
\begin{split}
\int_0^t \int_{B_R} u^*(x,s) \, \di x \di s 
&= \lim_{\mu \nearrow \mu^*}\int_0^t \int_{B_R}  u_\mu (x,s) \, \di x \di s \le C(\mu^*, g) t
\end{split}
\end{equation*}
for all $t>0$.
We thus obtain $\lim_{\mu \nearrow \mu^*}\|u_\mu - u^*\|_{L^1(B_R\times (0,t))} = 0$.
Similarly,
the monotone convergence theorem, H\"{o}lder inequality, and \eqref{eq:AE2} yield that
\begin{equation*}
\begin{split}
\int_0^t\int_{B_R} |x|^\sigma u^*(x,s)^p \, \di x \di s
\le C (1+t)^\frac{1}{p+1}\lim_{\mu \nearrow \mu^*} \left (\int_0^t\int_{B_R} |x|^\sigma u_\mu (x,s)^{p+1} \, \di x \di s \right)^\frac{p}{p+1}\\
\le C (1+t)^\frac{1}{p+1}\lim_{\mu \nearrow \mu^*}
\left[ J[\mu g]^\frac{2}{p+1} + J [\mu g]t\right]^\frac{p}{p+1}
\le C(\mu^*, g) (1+t)
\end{split}
\end{equation*}
for all $t>0$.
Therefore, $|\cdot|^\sigma (u^*)^p \in L^1 (B_R\times(0,t))$
for all $t>0$.
We thus obtain 
\[
\lim_{\mu \nearrow \mu^*} \||\cdot|^\sigma (u_\mu^p - (u^*)^p)\|_{L^1(B_R \times (0,t))} = 0 \quad \mbox{for all} \quad t>0.
\]
This is the condition \eqref{eq:L12}.

We shall check that $u^*$ satisfies the condition \eqref{eq:L11}.
In order to do that, we consider the following auxiliary problem
\begin{equation}
\label{eq:AP}
\left\{
\begin{array}{ll}
	\partial_t u  = \Delta u + |x|^{\sigma}(u^*)^p, \quad  &x\in B_R,\,\,\, t\in (0,\infty), \vspace{3pt}\\
	u = 0, \quad  &x\in  \partial B_R,\,\,\, t\in (0,\infty), \vspace{3pt}\\
        u(x,0) = \mu^*g(x), \quad  &x\in B_R. \vspace{3pt}\\
\end{array}
\right.
\end{equation}
By the standard theory of $L^1$-semigroup (see {\it e.g.}, \cite{QS19}), problem~\eqref{eq:AP} possesses an $L^1$-solution $v \in C([0,t]; L^1(B_R))$ provided that $|\cdot|^\sigma (u^*)^p \in L^1 (B_R\times(0,t))$.
Furthermore, it is an $L^1$-contracting mapping.
Then
\begin{equation*}
\begin{split}
&v(\cdot, t) =  e^{t\Delta}\mu^*g + \int_0^t e^{(t-s)\Delta} |\cdot|^\sigma u^*(\cdot, s)^{p} \, \di s,\\
&u_\mu(\cdot, t) =  e^{t\Delta}\mu g + \int_0^t e^{(t-s)\Delta} |\cdot|^\sigma u_\mu(\cdot, s)^{p} \, \di s.\\
\end{split}
\end{equation*}
Thus, for all $\mu\in (0,\mu^*)$, we have
\begin{equation*}
\begin{split}
&\|v(\cdot, t) - u_\mu(\cdot, t)\|_{L^1(B_R)} \\
&\le (\mu^*-\mu) \|e^{t\Delta} g\|_{L^1(B_R)} + \int_0^t \|e^{(t-s) \Delta}|\cdot|^\sigma(u^*(\dot, s)^p - u_\mu(\cdot, s)^p)\|_{L^1(B_R)} \, \di s\\
&\le  (\mu^*-\mu) \| g\|_{L^1(B_R)} 
+ \int_0^t \||\cdot|^\sigma(u^*(\cdot,s)^p - u_\mu(\cdot,s)^p)\|_{L^1(B_R)} \, \di s.
\end{split}
\end{equation*}
The dominated convergence theorem shows that 
the right-hand side of the above inequality tends to $0$ as $\mu \nearrow \mu^*$.
Hence, we have $u^* = v$. This implies that $u^*$ satisfies the condition \eqref{eq:L11}.
Since $v$ is an $L^1$-global-in-time solution, we conclude that $u^*$ is a minimal $L^1$-global-in-time solution.
The proof is complete.
\end{proof}


\bigskip
\noindent
{\bf Acknowledgement}\\
The authors would like to acknowledge K. Kumagai 
for providing the information on the transformation in the proof of Proposition~\ref{Proposition:SS}.
The work of Y. Seki was partly supported by Grant-in-Aid for scientific research (22K03387).\\

\noindent{\bf Conflict of interest}\\
The authors have no relevant financial or non-financial interests to disclose.\\

\noindent{\bf Data Availability}\\
Data sharing is not applicable to this article as no new data were created or analyzed in this study.


\end{document}